\documentclass{amsart}
\usepackage{amsfonts,amsthm,amsmath}
\usepackage{amssymb}
\usepackage{amscd}
\usepackage[utf8]{inputenc}
\usepackage[a4paper]{geometry}
\usepackage{enumerate}
\usepackage[usenames,dvipsnames]{color}
\usepackage{array}
\usepackage{array,tabularx}

\numberwithin{equation}{section}
\numberwithin{equation}{subsection}

\theoremstyle{plain}

\newtheorem{theorem}[equation]{Theorem}
\newtheorem{lemma}[equation]{Lemma}

\newtheorem{cor}[equation]{Corollary}

\theoremstyle{definition}

\newtheorem{remark}[equation]{Remark}

\newtheorem{definition}[equation]{Definition}

\def\C{\mathbb C}
\def\Q{\mathbb Q}

\def\Z{\mathbb Z}

\def\im{{\rm Im}}

\newcommand{\calv}{{\mathcal V}}

\newcommand{\calt}{{\mathcal T}}

\newcommand{\calO}{{\mathcal O}}

\newcommand{\calS}{{\mathcal S}}

\newcommand{\calL}{\mathcal{L}}

\newcommand{\tX}{\widetilde{X}}

\newcommand{\bP}{{\mathbb P}}

\newcommand{\bC}{{\mathbb C}}

\newcommand{\eca}{{\rm ECa}}
\newcommand{\pic}{{\rm Pic}}

\newcommand{\fr}{\mathfrak{r}}
\newcommand{\mfl}{\mathfrak{L}}

\newcommand{\bZ}{{\mathbb{Z}}}
\newcommand{\bQ}{{\mathbb{Q}}}

\author{J\'anos Nagy}
\address{Central European University, Dept. of Mathematics,  Budapest, Hungary}
\email{nagy\textunderscore janos@phd.ceu.edu}

\title{Cohomology of natural line bundles on generic normal surface singularities
}

\begin{document}

\thanks{The authors are partially supported by
 NKFIH Grant ``\'Elvonal (Frontier)'' KKP 126683.}

\begin{abstract}

Let $\mathcal{T}$ be an arbitrary resolution graph and $(X, 0)$ a generic complex analytic normal surface singularity, and $\tX$ a generic resolution corresponding to it.
Fix an effective integer cycle $Z$ supported on the exceptional curve and also an arbitrary Chern class $Z' \in L'$.

In this article we aim to compute the cohomology numbers $h^1(\calO_{Z}(Z'))$.
Notice, that the case $Z'_v < 0, v \in |Z|$ was discussed in \cite{NNA2}, where the main theorem was, that in this special case these cohomology numbers equal to the cohomology numbers of the generic line bundle in $\pic^{Z'}(Z)$
However the condition $Z'_v < 0, v \in |Z|$  was crucial in the proof and without this assumption the statement is far from being true.

In this article using the tecniques of relatively generic line bundles and relatively generic analytic structures from \cite{R} we give combinatorial algorithms to compute the cohomology numbers of natural line bundles $h^1(\calO_{Z}(Z'))$ for generic singularities in all cases.
\end{abstract}

\maketitle

\linespread{1.2}


\pagestyle{myheadings} \markboth{{\normalsize  J. Nagy}} {{\normalsize Cohomology of natural line bundles}}


\section{Introduction}\label{s:intr}

Let's have an arbitrary resolution graph $\mathcal{T}$ and a generic complex normal surface singularity with resolution $\tX$ corresponding to it in the sense explained in \cite{NNA2}.

The authors in \cite{NNA2} investigated the geometric genus and the analytical Poincaré series of the generic singularity $\tX$.
The key theorem towards the determination of these invariants was the following:

We fix a normal surface singularity $(X,o)$ and one of its good resolutions $\tX$ with
exceptional divisor $E$ and dual graph $\mathcal{T}$.

For any integral effective cycle $Z$ whose support $|Z|$ is included in 
 $E$ (but it can be smaller than $E$) write $\calv(|Z|)$ for the set of vertices 
 $\{v:\ |Z|=\sum_vE_v\}$ and $\calS'(|Z|) $ for the Lipman cone associated with
 the induced lattice $L(|Z|)$. 

Recall from \cite{NNA1}, that for any  $\tilde{l}\in - \calS'(|Z|)$ one has the Abel map $c^{\tilde{l}}(Z) :\eca ^{\tilde{l}}(Z)\to \pic^{\tilde{l}}(Z)$.

 By its definition, a line bundle $\calL\in \pic^{\tilde{l}}(Z)$ is in the image 
 $\im(c^{\tilde{l}}(Z))$ if and only if it has a section with no fixed components, that is, $H^0(Z,\calL)_{reg}\not=\emptyset $, where
$H^0(Z,\calL)_{reg}:=H^0(Z,\calL)\setminus \cup_v H^0(Z-E_v, \calL(-E_v))$.

For any $l'\in L'$ we denote the restriction of the natural 
line bundle $\calO_{\tX}(l') $ to $Z$ by $\calO_Z(l')$.

Denote also by $\tilde{l}$ the restriction $ R(l')$  of $l'\in L'$ into $ L'(|Z|)$, then we have the following theorem from \cite{NNA2}:

\begin{theorem}\label{th:CLB1}   
Assume that $(X,o)$ and its good resolution $(\tX,E)$ is generic and fix also some integer effective cycle $Z$ on it as above. 

(I) \  Assume that   $l'=\sum_{v\in \calv}l'_vE_v \in L'$ satisfies 
$l'_v <0$ for any $v\in\calv(|Z|)$ and $\tilde{l} = R(l') \in -\calS'(|Z|)$. Then 
  the following facts are equivalent:

(a) $\calO_Z(l')\in \im(c^{\tilde{l}}(Z))$, that is,   $H^0(Z,\calO_Z(l'))_{reg}\not=\emptyset $;

(b) $c^{\tilde{l}}(Z)$ is dominant, or equivalently,  $\calL_{gen}\in \im(c^{\tilde{l}}(Z))$, that is,
 $H^0(Z,\calL_{gen})_{reg}\not=\emptyset $,  
 for a generic line bundle $\calL_{gen}\in \pic^{\tilde{l}}(Z)$;

 (c) $\calO_Z(l')\in \im(c^{\tilde{l}}(Z))$, 
 and for any $D\in (c^{\tilde{l}}(Z))^{-1}(\calO_Z(l'))$ the tangent map
 $T_Dc^{\tilde{l}}(Z): T_D\eca^{\tilde{l}}(Z)\to T_{\calO_Z(l')}\pic^{\tilde{l}}(Z)$ is surjective.

 \vspace{1mm}

 \noindent (II)  Assume that   $l'=\sum_{v\in \calv}l'_vE_v \in L'$ such that 
$l'_v <0$ for any $v \in \calv(|Z|)$,  then $h^i(Z,\calO_Z(l'))=h^i(Z,\calL_{gen})$  for $i=0,1$ and a generic line bundle $\calL_{gen}\in \pic^{\tilde{l}}(Z)$.

\end{theorem}

Part (II) of the theorem above determines several cohomology numbers of natural line bundles on generic normal surface singularities, however the condition $l'_v <0$ for any $v\in\calv(|Z|)$ was cruical in the proof
and also the statement is far from being true without this condition.

To emphasise this, let's have another example of natural line bundles $\calO_Z(Z')$ on generic singularities, when $|Z| \cap |Z'| = \emptyset$ and $Z'_v = 1$ whenever $v$ is a vertex neighbour to $|Z|$.

In this case the computation of $h^1(\calO_{Z}(Z'))$ is equivalent to the computation of $\dim(\im(c^{R(Z')}(Z)))$, where $R(Z') = c_1(\calO_{Z}(Z'))$.

Indeed $\calO_{Z}(Z')$ is a generic line bundle in $\im(c^{R(Z')}(Z))$, so by \cite{NNA1} we have $h^1(\calO_{Z}(Z')) = h^1(Z) - \dim(\im(c^{R(Z')}(Z)))$.

The article of the author and A. Némethi \cite{NNAD} ivestigates the dimensions of images of Abel maps for arbitrary complex normal surface singularities giving algorithms to compute
them from cohomology numbers of cycles or periodic constants yielding the following result:

\begin{theorem}
Let's have an arbitrary complex normal surface singularity with resolution $\tX$, an integer effective cycle $Z\geq E$, and a Chern class $l' \in -S'$, then one has:

\begin{equation*} 
\dim(\im(c^{l'}(Z)))   = \min_{0\leq Z_1 \leq Z}\{\, (l', Z_1) + h^1(\calO_Z) - h^1(\calO_{Z_1})\, \}.
\end{equation*}
\end{theorem}

This theorem gives the following explicit combinatorial formulas for the special case of generic singularities:

\begin{cor} 
Assume that we have a resolution graph $\mathcal{T}$, and a generic resolution $\tX$ of a normal surface singularity with resolution graph $\mathcal{T}$.
Let's have an integral cycle $Z \geq E$ and an arbitrary Chern class $l' \in -S'$.

For any integer cycle $0\leq Z_1\leq Z$, let's write $E_{|Z_1|}$ for $\sum _{E_v\subset |Z_1|}E_v$, then we have:
\begin{equation*}\label{eq:gen}
\dim(\im(c^{l'}(Z))) = 1  -
 \min_{E \leq l \leq Z}\ \{ \chi(l)\} +  \min_{0\leq Z_1\leq Z}\big\{ \, (l', Z_1)  +  \min_{E_{|Z_1|} \leq l \leq Z_1 } \chi(l)
  - \chi(E_{|Z_1|})
 \, \big\}.
\end{equation*}
In particular, $\dim(\im(c^{l'}(Z))) $ is topological.
\end{cor}

Although these formulas are similar to the cohomology numbers of generic line bundles with given Chern class, they are not the same.

In this article, using the tecniques of relatively generic line bundles and relatively generic analytic structures we give combinatorial algorithms to compute the cohomology numbers of natural line bundles $h^1(\calO_{Z}(Z'))$ for generic singularities in all cases.

So the main result of the paper is the following:

\begin{theorem}
Let $\mathcal{T}$ be an arbitrary resolution graph, and let's have a generic resolution $\tX$ of a normal surface singularity with resolution graph $\mathcal{T}$.
Let's have an effective integer cycle $Z$, and an arbitrary Chern class $Z' \in L'$, then the cohomology numbers of the natural line bundle $h^1(\calO_{Z}(Z'))$ are
combinatorially computable from the resolution graph.
\end{theorem}

\section{Prelinimaries}

\subsection{The resolution}\label{ss:notation}
Let $(X,o)$ be the germ of a complex analytic normal surface singularity,
 and let us fix  a good resolution  $\phi:\widetilde{X}\to X$ of $(X,o)$.
We denote the exceptional curve $\phi^{-1}(0)$ by $E$, and let $\cup_{v\in\calv}E_v$ be
its irreducible components.

 Set also $E_I:=\sum_{v\in I}E_v$ for any subset $I\subset \calv$.
For the cycle $l=\sum n_vE_v$ let its support be $|l|=\cup_{n_v\not=0}E_v$.
For more details see \cite{trieste,NCL,Nfive}.
\subsection{Topological invariants}\label{ss:topol}
Let $\calt$ be the dual resolution graph
associated with $\phi$;  it  is a connected graph.
Then $M:=\partial \widetilde{X}$ can be identified with the link of $(X,o)$, it is also
an oriented  plumbed 3--manifold associated with $\calt$.
We will assume that  $M$ is a rational homology sphere,
or, equivalently,  $\mathcal{T}$ is a tree and all genus
decorations of $\mathcal{T}$ are zero. We use the same
notation $\mathcal{V}$ for the set of vertices, and $\delta_v$ for the valency of a vertex $v$.

$L:=H_2(\widetilde{X},\mathbb{Z})$, endowed
with a negative definite intersection form  $I=(\,,\,)$, is a lattice. It is
freely generated by the classes of 2--spheres $\{E_v\}_{v\in\mathcal{V}}$.
 The dual lattice $L':=H^2(\widetilde{X},\mathbb{Z})$ is generated
by the (anti)dual classes $\{E^*_v\}_{v\in\mathcal{V}}$ defined
by $(E^{*}_{v},E_{w})=-\delta_{vw}$, the opposite of the Kronecker symbol.
The intersection form embeds $L$ into $L'$. Then $H_1(M,\mathbb{Z})\simeq L'/L$, abridged by $H$.
Usually one also identifies $L'$ with those rational cycles $l'\in L\otimes \Q$ for which
$(l',L)\in\Z$, or, $L'={\rm Hom}_\Z(L,\Z)$.

Each class $h\in H=L'/L$ has a unique representative $r_h=\sum_vr_vE_v\in L'$ in the semi-open cube
(i.e. each $r_v\in \bQ\cap [0,1)$), such that its class  $[r_h]$ is $h$.

All the $E_v$--coordinates of any $E^*_u$ are strict positive.
We define the Lipman cone as $\calS':=\{l'\in L'\,:\, (l', E_v)\leq 0 \ \mbox{for all $v$}\}$.
It is generated over $\bZ_{\geq 0}$ by $\{E^*_v\}_v$.

\subsection{Analytic invariants}\label{ss:analinv}
{\bf The group ${\rm Pic}(\widetilde{X})$}
of  isomorphism classes of analytic line bundles on $\widetilde{X}$ appears in the exact sequence
\begin{equation}\label{eq:PIC}
0\to {\rm Pic}^0(\widetilde{X})\to {\rm Pic}(\widetilde{X})\stackrel{c_1}
{\longrightarrow} L'\to 0, \end{equation}
where  $c_1$ denotes the first Chern class. Here
$ {\rm Pic}^0(\widetilde{X})=H^1(\widetilde{X},\calO_{\widetilde{X}})\simeq
\C^{p_g}$, where $p_g$ is the {\it geometric genus} of
$(X,o)$. $(X,o)$ is called {\it rational} if $p_g(X,o)=0$.
 Artin characterized rationality topologically
via the graphs; such graphs are called `rational'. By this criterion, $\calt$
is rational if and only if $\chi(l)\geq 1$ for any effective non--zero cycle $l\in L_{>0}$.
Here $\chi(l)=-(l,l-Z_K)/2$, where $Z_K\in L'$ is the (anti)canonical cycle
identified by adjunction formulae
$(-Z_K+E_v,E_v)+2=0$ for all $v$.

The epimorphism
$c_1$ admits a unique group homomorphism section $l'\mapsto s(l')\in {\rm Pic}(\widetilde{X})$,
 which extends the natural
section $l\mapsto \calO_{\widetilde{X}}(l)$ valid for integral cycles $l\in L$, and
such that $c_1(s(l'))=l'$  \cite{trieste,OkumaRat}.
We call $s(l')$ the  {\it natural line bundles} on $\widetilde{X}$ and we denote it by $\calO_{\widetilde{X}}(l')$.
By  their definition, $\calL$ is natural if and only if some power $\calL^{\otimes n}$
of it has the form $\calO_{\tX}(-l)$ for some $l\in L$.

Furthermore for an arbitrary effective non--zero integral cycle $Z\in L_{>0}$ let's denote the restriction of the line bundle $\calO_{\widetilde{X}}(l')$ to the cycle $Z$
by $\calO_{Z}(l')$.

If we denote the $*$-restriction map by $R : L' \to L'_{|Z|}$, then we have $c_1(\calO_{Z}(l')) = R(l')$.

\subsection{Rational line bundles}

In the following let's introduce the notation of rational line bundles on an effective cycle $Z$, which is supported on the resolution of a normal surface singularity $\tX$:

\begin{definition}
Let's have a normal surface singularity with resolution $\tX$ and an effective integer cycle $Z > 0$ on it and furthermore $l'' \in L'_{|Z|} \otimes \bQ$.

A rational line bundle on $Z$ with Chern class $l''$ is an equivalence class of a pair of an integer and a line bundle $(N, \calL)$, such that $N \cdot l'' \in L'_{|Z|}$ and $\calL \in \pic^{N \cdot l''}(Z)$
and we say, that the two pairs $(N_1, \calL_1)$ and $(N_2, \calL_2)$ are equivalent if $N_2 \cdot \calL_1 \cong N_1 \cdot \calL_2$.

We call $l''$ the Chern class of the rational line bundle, and we denote the set of rational line bundles with Chern class $l''$ by $\pic^{l''}(Z)$.
If $\calL \in \pic^{l''}(Z)$, we denote $c^1(\calL) = l''$.
\end{definition}

Since the Picard groups $\pic^{l'}(Z), l' \in L'$ are torsion free and are isomorphic to $H^1(\calO_Z)$ as affine spaces, we get that for any $l'' \in L'_{|Z|} \otimes \bQ$ we have $\pic^{l''}(Z) \cong H^1(\calO_Z)$.

If we have two rational line bundles $\calL_1 \in \pic^{l''_1}(Z)$ and $\calL_2 \in \pic^{l''_2}(Z)$, then define $\calL_1 \otimes \calL_2 \in \pic^{l''_1 + l''_2}(Z)$ in the following way:

If $\calL_1$ is represented by $(N_1, \calL_{s, 1})$ and $\calL_2$ is represented by $(N_2, \calL_{s, 2})$ , then $\calL_1 \otimes \calL_2$ is  represented by $(N_2 \cdot N_1, N_1 \cdot \calL_{s, 2} \otimes N_2 \cdot \calL_{s, 1} )$.
It's easy to see, that the equivalence class of this pair is independent of the representations of the rational line bundles $\calL_1, \calL_2$.

Similarly we can define $\calL^{-1}$ and $t \cdot \calL$ for any rational number $t \in \bQ$ and any rational line bundle $\calL$ with Chern classes $-l''$ and $t \cdot l''$ respectively.

If we have any divisor $D \in \eca^{l'}(Z)$ for a Chern class and $r \in \bQ$ is a rational number, then the pair $(N, \calO_Z(Nr \cdot D))$ defines a rational line bundle, 
if $Nr \in \bZ$, and we denote the corresponding rational line bundle by $\calO_Z(r \cdot D)$.

\subsection{Minkowski sum of affine varieties}

Let's use the notation in the article, that if $X$ and $Y$ are two subsets of a complex vector space $\bC^{N}$, then we denote by $X \oplus Y$ the Minkowski sum of
the two subsets.

Notice, that if $X \subset \bC^N$ is some irreducible analytic subvariety and $ Y \in \bC^N$ is some other irreducible analytic subvariety, then $\dim( Y \oplus X) = \dim(X)$, if and only if $Y \oplus X = p \oplus X$ for some $p \in Y$, if and only if $ A(Y) \oplus X = p \oplus X$ for some $p \in Y$, where $A(Y)$ is the affine hull of $Y$.

For more about minkowski sums of affine varieties look at \cite{Min}.

\section{Effective Cartier divisors and Abel maps}

\subsection{}  Let
$\eca(Z)$  be the space of effective Cartier divisors on 
$Z$ introduced in \cite{NNA1}.
 Their support is zero--dimensional in $E$.

Taking the class of a Cartier divisor provides  a map
$c:\eca(Z)\to \pic(Z)$, which we call the Abel map.

Let
$\eca^{l'}(Z)$ be the set of effective Cartier divisors with
Chern class $l'\in L'$, that is,
$\eca^{l'}(Z):=c^{-1}(\pic^{l'}(Z))$.
For any $Z_2\geq Z_1>0$ one has the commutative diagram
\begin{equation}\label{eq:diagr}
\begin{picture}(200,45)(0,0)
\put(50,37){\makebox(0,0)[l]{$
\eca^{l'}(Z_2)\,\longrightarrow \, \pic^{l'}(Z_2)$}}
\put(50,8){\makebox(0,0)[l]{$
\eca^{l'}(Z_1)\,\longrightarrow \, \pic^{l'}(Z_1)$}}
\put(70,22){\makebox(0,0){$\downarrow$}}
\put(135,22){\makebox(0,0){$\downarrow$}}
\end{picture}
\end{equation}

Let us fix  $Z\in L$, $Z\geq E$. (The restriction $Z\geq E$ is imposed by the
easement of the presentation, everything can be adopted  for $Z>0$).

As usual, we say that $\calL\in \pic^{l'}(Z)$ has no fixed components if
\begin{equation}\label{eq:H_0}
H^0(Z,\calL)_{reg}:=H^0(Z,\calL)\setminus \bigcup_{v \in |Z|} H^0(Z-E_v, \calL(-E_v))
\end{equation}
is non--empty. 

Note that $H^0(Z,\calL)$ is a module over the algebra
$H^0(\calO_Z)$, hence one has a natural action of $H^0(\calO_Z^*)$ on
$H^0(Z, \calL)_{reg}$. For the next lemma see e.g. \cite[\S 3]{Kl}.

\begin{lemma}\label{lem:H_0} Consider the restriction of $c$, $c^{l'}:\eca^{l'}(Z)
\to \pic^{l'}(Z)$. Then $\calL$ is in the image of $c$ if and only if
$H^0(Z,\calL)_{reg}\not=\emptyset$. In this case, $c^{-1}(\calL)=H^0(Z,\calL)_{reg}/H^0(\calO_Z^*)$.
\end{lemma}

Note that $H^0(Z,\calL)_{reg}\not=\emptyset \ \Rightarrow\  l'\in -\calS'$, conversely, if $l'=-\sum_vm_vE^*_v\in -\calS'$, for certain $m_v\in\bZ_{\geq 0}$, then one can construct
for each $E_v$ cuts in $\widetilde{X}$
intersecting $E_v$ in a generic point and having with  it intersection multiplicity $m_v$. Their collection
$D$ provides an element in $\eca^{l'}(Z)$ whose image by $c$ is
 $\calO_Z(D)\in\pic^{l'}(Z)$. Therefore
\begin{equation}\label{eq:empty}
\eca^{l'}(Z)\not =\emptyset \ \ \Leftrightarrow \ \ l'\in -\calS'.
\end{equation}

The action of  $H^0(\mathcal{O}^{*}_{Z})$ can be analysed quite explicitly.

 Note that from the exact sequence
$ 0 \to H^0(\mathcal{O}_{Z-E}(-E)) \to H^0(\mathcal{O}_{Z}) \stackrel{r_E}{\longrightarrow}
 H^0(\mathcal{O}_{E})=\bC\to 0$
one gets that
  $ H^0(\mathcal{O}^{*}_{Z}) = r_E^{-1}(\bC^*)=
  H^0(\mathcal{O}_{Z}) \setminus H^0(\mathcal{O}_{Z-E}(-E))$.
 In particular, the projectivized $\bP H^0(\calO^*_Z)$,
 as algebraic group, is isomorphic with the vector space $H^0(\mathcal{O}_{Z-E}(-E))$,
 and $H^0(Z,\calL)_{reg}/H^0(\calO^*_Z)=\bP H^0(Z,\calL)_{reg}/\bP H^0(\calO^*_Z)$.

We have the following lemma and theorem from \cite{NNA1}:

 \begin{lemma}\label{lem:free} Assume that $H^0 (Z,\calL)_{reg}\not=\emptyset$. Then

 (a) the action of $H^0(\calO_Z^*)$ on $H^0(Z,\calL)_{reg}$ is algebraic and free, and

 (b) $ H^0(Z,\calL)_{reg}$ over $H^0(Z,\calL)_{reg}/H^0(\calO_Z^*)$ is a principal  bundle, or, equivalently,
 $\bP H^0(Z,\calL)_{reg}$ over $\bP H^0(Z,\calL)_{reg}/\bP H^0(\calO_Z^*)$ is a principal affine bundle.

 \noindent
 Hence, the fiber $c^{-1}(\calL)$, $\calL\in \im (c^{l'})$,
 is a smooth, irreducible quasiprojective variety of
  dimension
 \begin{equation}\label{eq:dimfiber}
 h^0(Z,\calL)-h^0(\calO_Z)=
 (l',Z)+h^1(Z,\calL)-h^1(\calO_Z).
 \end{equation}
 \end{lemma}


 \begin{theorem}\label{th:smooth} If $l'\in-\calS'$ then the following facts hold.

  (1)  $\eca^{l'}(Z)$ is a smooth variety of dimension $(l',Z)$.

  (2) The natural restriction map $r:\eca^{l'}(Z)\to \eca^{l'}(E)$ is a
  locally trivial  fiber bundle with fiber isomorphic to an affine space. Hence,
 the homotopy type of $\eca^{l'}(Z)$ is independent of the choice of $Z$ and
 it depends only on the topology of $(X,o)$.
 \end{theorem}

Consider again an integer effective cycle $Z $, and a Chern class $l'\in-\calS'$ associated with a resolution $\tX$, as above. 

Then, besides the Abel map $c^{l'}(Z)$ one can consider its `multiples' $\{c^{nl'}(Z)\}_{n\geq 1}$.
 It turns out that $n\mapsto \dim \im (c^{nl'}(Z))$ is a non-decreasing sequence, which stabilises after a while.

The image  $\im (c^{nl'}(Z))$ is an affine subspace for $n\gg 1$, whose dimension $e_Z(l')$ is independent of $n\gg 0$, and essentially it depends only
on the $E^*$--support of $l'$ (i.e., on $I\subset \calv$, where $-l'=\sum_{v\in I}a_vE^*_v$ with all
$\{a_v\}_{v\in I}$ nonzero).
The statement  $e_Z(l')=e_Z(I)$ plays a crucial role in different analytic properties of $\tX$
(surgery formula, $h^1(\calL)$--computations, base point freeness properties). For details see \cite{NNA1}.

If $v \in \calv$ is an arbitrary vertex and $n$ is a large integer, then $\im (c^{-n E_v^*}(Z))$ is an affine subspace and it is parallel to a linear subspace, 
which we denote by $V_v(Z)$ and this subspace is independent of the chosen integer $n$.

Similarly if $l' \in -S'$ with $|l'| = I$, and $n$ is a large integer, then $\im (c^{-n l'}(Z))$ is an affine subspace and it is parallel to a linear subspace, 
which we denote by $V_I(Z)$ and this subspace is independent of the chosen integer $n$.

For more about the subspaces $V_v(Z), V_I(Z)$ see \cite{NNA1}.

\section{Relatively generic analytic structures on surface singularities}

In this section we wish to summarise the results from \cite{R} about relatively generic analytic structures we need in this article. 

\subsection{The relative setup.}

We consider an integer cycle $Z$ on a resolution $\tX$ with resolution graph $\mathcal{T}$, and a smaller cycle $Z_1 \leq Z$, where we denote $|Z_1| = \calv_1$ and the subgraph corresponding to it by $\mathcal{T}_1$.

We have the restriction map $r: \pic(Z)\to \pic(Z_1)$ and one has also the (cohomological) restriction operator
  $R_1 : L'(\mathcal{T}) \to L_1':=L'(\mathcal{T}_1)$
(defined as $R_1(E^*_v(\mathcal{T}))=E^*_v(\mathcal{T}_1)$ if $v\in \calv_1$, and
$R_1(E^*_v(\mathcal{T}))=0$ otherwise).

For any $\calL\in \pic(Z)$ and any $l'\in L'(\mathcal{T})$ it satisfies
\begin{equation*}
c_1(r(\calL))=R_1(c_1(\calL)).
\end{equation*}

In particular,
we have the following commutative diagram as well:

\begin{equation*}  
\begin{picture}(200,40)(30,0)
\put(50,37){\makebox(0,0)[l]{$
\ \ \eca^{l'}(Z)\ \ \ \ \ \stackrel{c^{l'}(Z)}{\longrightarrow} \ \ \ \pic^{l'}(Z)$}}
\put(50,8){\makebox(0,0)[l]{$
\eca^{R_1(l')}(Z_1)\ \ \stackrel{c^{R_1(l')}(Z_1)}{\longrightarrow} \  \pic^{R_1(l')}(Z_1)$}}
\put(162,22){\makebox(0,0){$\downarrow \, $\tiny{$r$}}}
\put(78,22){\makebox(0,0){$\downarrow \, $\tiny{$\fr$}}}
\end{picture}
\end{equation*}

By the `relative case' we mean that instead of the `total' Abel map
$c^{l'}(Z)$ we study its restriction above a fixed fiber of $r$.

That is, we fix some  $\mfl\in \pic^{R_1(l')}(Z_1)$, and we study
the restriction of $c^{l'}(Z)$ to $(r\circ c^{l'}(Z))^{-1}(\mfl)\to r^{-1}(\mfl)$.

 The subvariety $(r\circ c^{l'}(Z))^{-1}(\mfl)
=(c^{R_1(l')}(Z_1) \circ \fr)^{-1}(\mfl) \subset \eca^{l'}(Z)$ is denoted by $\eca^{l', \mfl}$.

\begin{theorem}\label{cor:smoothirreddim}
Fix $l'\in -\calS'$, $Z\geq E$, $Z_1 \leq Z$ and  $\mfl\in \pic^{R_1(l')}(Z_1)$ and assume, that  $\eca^{l', \mfl}$ is nonempty. 
Then it is smooth of dimension $h^1(Z_1,\mfl)  - h^1(Z_1,\calO_{Z_1})+ (l', Z)$ and it is irreducible.
\end{theorem}

Let' recall from \cite{R} the analouge of the theroems about dominance of Abel maps in the relative setup:

\begin{definition}
Fix $l'\in -\calS'$,
$Z\geq E$, $Z_1 \leq Z$ and  $\mfl\in \pic^{R_1(l')}(Z_1)$ as above.
We say that the pair $(l',\mfl ) $ is {\it relative
dominant} if the closure of $ r^{-1}(\mfl)\cap {\rm Im}(c^{l'}(Z))$ in  $r^{-1}(\mfl)$
is $r^{-1}(\mfl)$.
\end{definition}

\begin{theorem}\label{th:dominantrel} One has the following facts:

(1) If $(l',\mfl)$ is relative dominant then $ \eca^{l', \mfl}$ is
nonempty and $h^1(Z,\calL)= h^1(Z_1,\mfl)$ for any
generic line bundle $\calL\in r^{-1}(\mfl)$.

(2) $(l',\mfl)$ is relative dominant  if and only if for all
 $0<l\leq Z$, $l\in L$ one has
$$\chi(-l')- h^1(Z_1, \mfl) < \chi(-l'+l)-
 h^1((Z-l)_1, (\mfl | (Z-l)_1)  \otimes \calO_{(Z-l)_1}(-l) ).$$, where we denote $(Z-l)_1 = \min(Z-l, Z_1)$.
\end{theorem}

\begin{theorem}
Fix $l'\in -\calS'$, $Z > 0$ , $0 \leq Z_1 \leq Z$ and  $\mfl \in \pic^{R_1(l')}(Z_1)$  as in Theorem \ref{th:dominantrel}. 
Then for any $\calL\in r^{-1}(\mfl)$ one has
\begin{equation*}\label{eq:genericLrel}
\begin{array}{ll}h^1(Z,\calL)\geq \chi(-l')-
\min_{0\leq l\leq Z,\ l\in L} \{\,\chi(-l'+l) -
 h^1((Z-l)_1, (\mfl | (Z-l)_1)  \otimes \calO_{(Z-l)_1}(-l) )\, \}, \\
h^0(Z,\calL)\geq \max_{0\leq l\leq Z,\, l\in L}
\{\,\chi(Z-l,\calL(-l))+   h^1((Z-l)_1, (\mfl | (Z-l)_1)  \otimes \calO_{(Z-l)_1}(-l) )\,\}.\end{array}\end{equation*}
Furthermore, if $\calL$ is generic in $r^{-1}(\mfl)$
then in both inequalities we have equalities and we have even the bit stronger statement, that $h^0(Z,\calL) =
\max_{0\leq l\leq Z,\, l\in L, H^0(Z-l, \calL(-l))_0 \neq \emptyset}\{\,\chi(Z-l,\calL(-l))+   h^1((Z-l)_1, (\mfl | (Z-l)_1)  \otimes \calO_{(Z-l)_1}(-l) )\,\}$.
\end{theorem}

In the following we recall the results from \cite{R} about relatively generic analytic structures:

Let's fix a a topological type, so a resolution graph $\mathcal{T}$ with vertex set $\calv$.

We consider a partition $\calv = \calv_1 \cup  \calv_2$ of the set of vertices $\calv=\calv(\mathcal{T})$. They define
two (not necessarily connected) subgraphs $\mathcal{T}_1$ and $\mathcal{T}_2$.
We call the intersection of an exceptional divisor from
$ \calv_1 $ with an exceptional divisor from  $ \calv_2 $ a
{\it contact point}. 

For any $Z\in L=L(\mathcal{T})$ we write $Z=Z_1+Z_2$,
where $Z_i\in L(\mathcal{T}_i)$ is
supported in $\mathcal{T}_i$ ($i=1,2$).
Furthermore, parallel to the restriction
$r_i : \pic(Z)\to \pic(Z_i)$ one also has the (cohomological) restriction operator
  $R_i : L'(\mathcal{T}) \to L_i':=L'(\mathcal{T}_i)$
(defined as $R_i(E^*_v(\mathcal{T}))=E^*_v(\mathcal{T}_i)$ if $v\in \calv_i$, and
$R_i(E^*_v(\mathcal{T}))=0$ otherwise).

For any $l'\in L'(\mathcal{T})$ and any $\calL\in \pic^{l'}(Z)$ it satisfies:

\begin{equation*}\label{eq:CHERNREST}
c_1(r_i(\calL))=R_i(c_1(\calL)) | L'_{|Z_1|}.
\end{equation*}

In the following for the sake of simplicity we will denote $r = r_1$ and $R = R_1$.

Furthermore let's have a fixed analytic type $\tX_1$ for the subtree $\mathcal{T}_1$ (if it is disconnected, then an analytic type for each connected component).

Also for each vertex $v_2 \in \calv_2$ which has got a neighbour $v_1$ in $\calv_1$ we fix a cut $D_{v_2}$ on $\tX_1$, along we glue the exceptional divisor $E_{v_2}$.
This means, that $D_{v_2}$ is a divisor, which intersects the exceptional divisor $E_{v_1}$ transversally in one point and we will glue the exceptional divisor $E_{v_2}$ in a way, 
such that $E_{v_2} \cap \tX_1$ equals $D_{v_2}$.

If for some vertex $v_2 \in \calv_2$, which has got a neighbour in $\calv_1$ we don't say explicitely what is the fixed cut, then it should be understood in the way, that we glue the exceptional divisor $E_{v_2}$ along a generic cut.

Let's plumb the tubular neihgbourhoods of the exceptional divisors $E_{v_2}, v_2 \in \calv_2$  with the above conditions generically to the fixed resolution $\tX_1$, now we get a big singularity $\tX$ and we say that $\tX$ is a relatively generic singularity corresponding to the analytical structure $\tX_1$ and the cuts $D_{v_2}$, for the precise explanation of genericity look at \cite{R}.

We have the following theorem with this setup from \cite{R}:

\begin{theorem}\label{relgen}
Let's have the setup as above, so two resolution graphs $\mathcal{T}_1 \subset \mathcal{T} $ with vertex sets $\calv_1 \subset \calv$, where $\calv = \calv_1 \cup \calv_2$  and a fixed singularity $\tX_1$ for the resolution graph $\mathcal{T}_1$, and cuts $D_{v_2}$ along we glue $E_{v_2}$ for all vertices $v_2 \in \calv_2$, which has got a neighbour in $\calv_1$.

Now let's assume that $\tX$ has a relatively generic analytic stucture on $\mathcal{T}$ corresponding to $\tX_1$.

Furthermore let's have an effective cycle $Z$ on $\tX$ and let's have $Z = Z_1 + Z_2$, where $|Z_1| \subset \calv_1$ and $|Z_2| \subset \calv_2$.

1) 
Let's have a natural line bundle $\calL$ on $\tX$, such that $c_1(\calL)= l' = - \sum_{v \in \calv} a_v E_v$, with $a_v > 0, v \in \calv_2 \cap |Z|$, and let's denote $c_1 (\calL | Z) = l'_ m$, furthermore let's denote $\mfl = \calL | Z_1$, then we have the following:

We have $H^0(Z,\calL)_0\not=\emptyset$ if and only if $(l',\mfl)$ is relative dominant on the cycle $Z$ or equivalently:

\begin{equation*}
\chi(-l')- h^1(Z_1, \mfl) < \chi(-l'+l)-  h^1((Z-l)_1, (\mfl | (Z-l)_1)  \otimes \calO_{(Z-l)_1}(-l) ),
\end{equation*}
for all $0 < l \leq Z$.

2) 

Let's have a natural line bundle $\calL$ on $\tX$, such that $c_1(\calL)= l' = - \sum_{v \in \calv} a_v E_v$, with $a_v > 0, v \in \calv_2 \cap |Z|$, and let's denote $c_1 (\calL | Z) = l'_ m$, furthermore let's denote $\mfl = \calL | Z_1$, then we have the following:

\begin{equation*}
h^1(Z, \calL) =  h^1(Z, \calL_{gen}),                                   
\end{equation*}
where $\calL_{gen}$ is a generic line bundle in $ r^{-1}(\mfl) \subset \pic^{l'_m}(Z)$, or equivalently:

\begin{equation*}
h^1(Z, \calL)= \chi(-l') - \min_{0 \leq l \leq Z}(\chi(-l'+l)-  h^1((Z-l)_1, (\mfl | (Z-l)_1)  \otimes \calO_{(Z-l)_1}(-l) )).
\end{equation*}

3) 

Let's have a natural line bundle $\calL$ on $\tX$, such that $c_1(\calL)= l' = - \sum_{v \in \calv'} a_v E_v$, and assume, that  $a_v \neq 0$ if $ v \in \calv_2 \cap |Z|$.
Let's denote $c_1 (\calL | Z) = l'_ m$, furthermore let's denote $\mfl = \calL | Z_1$.

Assume, that $H^0(Z,\calL)_0 \not=\emptyset$, and pick an arbitrary $D \in (c^{l'_m}(Z))^{-1}\calL \subset \eca^{l'_m, \mfl}$.
Then $ c^{l'_m}(Z) : \eca^{l'_m, \mfl} \to r^ {-1}(\mfl  )$ is a submersion in $D$, and $h^1(Z,\calL) = h^1(Z_1, \mfl)$.

In particular the map $ c^{l'_m}(Z) :  \eca^{l'_m, \mfl} \to r^ {-1}(\mfl  )$ is dominant, which means $(l'_m,\mfl)$ is relative dominant on the cycle $Z$, or equivalently:

\begin{equation*}
\chi(-l')- h^1(Z_1, \mfl) < \chi(-l'+l)-  h^1((Z-l)_1, (\mfl | (Z-l)_1)  \otimes \calO_{(Z-l)_1}(-l) )),
\end{equation*}
for all $0 < l \leq Z$.
\end{theorem}

\begin{remark}

In the theorem above in any formula one can replace $l'$ with $l'_m$, since for every $0 \leq l \leq Z$ one has $\chi(-l')- \chi(-l'+l) = \chi(-l'_m)- \chi(-l'_m+l) = -(l', l) - \chi(l)$.

\end{remark}

\section{cohomology of natural line bundles on generic singularities}

In the following we aim to compute all the cohomology numbers of restrictions of natural line bundles $h^1(\calO_Z(Z'))$, where $Z$ is an arbitrary integer effective cycle and $Z'\in L'$ is an arbitrary Chern class on a generic singularity with resolution $\tX$ correspongind to the fixed resolution graph $\mathcal{T}$.

Notice, that computing the dimensions $d_{Z, l'}$ is a special case of computing these numbers as explained in the introduction, however in this general situation we will not get an explicit expression, just a combinatorial algorithm which computes these numbers.

We start the discussion with some lemmas:

First we have the following possibly folklore lemma, however for the sake of completeness we prove it.

\begin{lemma}
Let $C$ be a probably non compact, irreducible, smooth complex curve, and let's have an analytic map $ f : C \to \bC^N$ for some integer $N \geq 0$, such that the affine hull
of the image $f(C)$ is the whole affine space $\bC^N$.

Let's look at the map $ f^d : S^d(C) \to \bC^N$, where $S^d(C)$ is the $d$-fold symmetric product, which is defined by $f^d(p_1, \cdots p_d) = \sum_{1 \leq i \leq d} f(p_i)$.

Now, if $p$ is a generic point of $C$, then the rank of the map $f^d$ in $d \cdot (p)$ is $ \min(d, N)$, which is equal to the rank of $ f^d$ in a generic point of $ S^d(C)$.
\end{lemma}
\begin{proof}
The statement, that the rank of $ f^d$ in a generic point of $ S^d(C)$ is $ \min(d, N)$ is clear, since the dimension of $f^d( S^d(C))$ is $ \min(d, N)$.

Indeed, we have $f^d( S^d(C)) = \oplus_{1 \leq i \leq d} f(C)$, from which we immediately get that $\dim(f^d( S^d(C))) \leq  \min(d, N)$.

On the other hand, since the  affine hull of the image $f(C)$ is the whole affine space $\bC^N$, we can find $j = \min(d, N)$ different generic smooth points $p_1, \cdots, p_j \in f(C)$,
such that the tangent maps $T_{p_i}(f(C))$ are linearly independent.
 We get that the tangent space of $S^d(C)$ at its generic point $\sum_{1 \leq i \leq j}p_i$ is at least $j$ dimensional, so we get indeed $\dim(f^d( S^d(C))) = \min(d, N)$.

Observe, that it's enough to prove the statement when $d \leq N$, and in this case we have to prove, that the tangent map of the map $f^d$ in $d \cdot (p)$ is injective.

Indeed if $d > N$ and $p \in C$ is a generic point, then we have the map $g : S^N(C) \to S^d(C)$ given by $g(x_1, \cdots, x_N) = ((d-N) \cdot (p), x_1, \cdots, x_N)$ and
since we know, that the tangent map of the map $f^N : S^N(C) \to \bC^N$ is surjective, we indeed get, that the tangent map of the map $f^d : S^N(C) \to \bC^N$ is also surjective.

Now let $p$ be a generic point of $C$ and let $z$ be a local holomorphic coordinate on $C$ near $z$.

Let $(z_1, \cdots z_d)$ be the corresponding holomorphic coordinates on $C^d$ near $(p, p, ..., p)$, then the elementary symmetric polynomials in $(z_1, \cdots z_d)$, $\sigma_1, \sigma_2, ..., \sigma_d$ are holomorphic coordinates on $S^d(C)$.

Write the map $f: C \to \bC^N$ locally like $f(z) = (f_1(z), \cdots f_N(z))$, where the functions $f_i$ are holomorphic.

We can write the symmetric funtions $f_i(z_1) + \cdots + f_i(z_d) = g_i(\sigma_1,  \sigma_2, ..., \sigma_d)$ in terms of the elementary symmetric polynomials.

Observe, that if $i \leq d$, then $ \sum_{1 \leq j \leq d} z_j^i = a_i \sigma_i + G_i(\sigma_1,  \sigma_2, ..., \sigma_d)$, where $a_i$ is a nonzero constant, and $G_i$ hasn't got linear terms in the variables $\sigma_1,  \sigma_2, ..., \sigma_d$.

It means exactly, that the tangent map $T_{d \cdot (p)}S^d(C) \to T_{f^d(d \cdot (p))}\bC^N$ is injective, if the vectors $\frac{\delta}{\delta_z}f, \cdots (\frac{\delta}{\delta_z})^d f$ are linearly independent.

So we have to prove, that in a generic point $p \in C$ the vectors $\frac{\delta}{\delta_z}f, \cdots (\frac{\delta}{\delta_z})^d f$ are linearly independent.

Assume to the contrary, that there are holomorphic functions $l_1, \cdots l_d$ on an open subset $U \subset C$, such that $\sum_{1 \leq i \leq d} l_i(z) \cdot (\frac{\delta}{\delta_z})^i f = 0$, such that not all of the functions $l_1, \cdots l_d$ are constant $0$.

It means, that for every $1 \leq j \leq N$ one has $\sum_{1 \leq i \leq d} l_i(z) \cdot (\frac{\delta}{\delta_z})^i f_j = 0$.

It means, that $\frac{\delta}{\delta_z} f_j$ are solutions of a $d-1$ order homogenous differential equtaion, however such a linear homogenous differential equation can have at most $d-1$ linearly independent solutions locally.
This means, that, there exist constants $b_1, \cdots b_N \in \bC$, such that $\sum_{1 \geq i \geq N} b_i \frac{\delta}{\delta_z} f_j = 0$ on $U$, however this means, that for every smooth point $q \in f(C)$ one has $T_p(f(C)) \in H$, where H is a fixed hyperplane in $\bC^N$.

However this means, that $f(C)$ is contained in an affine hyperplane parallel to $H$, which contradicts the fact, that the affine hull of $f(C)$ is $\bC^N$.
\end{proof}

Now let's prove the following key lemma:

\begin{lemma}
Let $\mathcal{T}$ be an arbitrary resolution graph, and let's have a generic singularity with resolution $\tX$ corresponding to it, $I \subset \calv$ an arbitrary subset and $Z \geq E$ an effective cycle, such that $Z_v = 1$ for all $v \in I$.

For all vertices $v \in I$ let's have integers $r_v, n_v \geq 0$, $r_v + n_v > 0$, and rational numbers $a_{v,1}, \cdots , a_{v, r_v} > 0$ ,  $b_{v, 1}, \cdots , b_{v, n_v} > 0$, $a_v =  \sum_{1 \leq i \leq r_v} a_{v, i}$, $b_v =  \sum_{1 \leq i \leq n_v} b_{v, i}$, such that with the notation $m_v = a_v - b_v$ we have $m_v \geq 0$ and $m_v \in \bZ$.

Let's denote $l' = \sum_{v \in I} -m_v E_v^*$  and let's have $d_{Z, l'} = \dim(\im(c^{l'}(Z)))$.

Let's have the subset $I' \subset I$ consisting of vertices $v$, such that $ \dim(V_v(Z) \oplus (\im(c^{l'}(Z))) = \dim(\im(c^{l'}(Z))$ holds, or $n_v = 0$ and $a_{v, i} \in \bZ$ holds for all $1 \leq i \leq r_v$.

Pick generic ponts on the exceptional divisor $E_v$, $p_{v, 1}, \cdots ,p_{v, r_v}$ and $q_{v, 1}, \cdots ,q_{v, n_v}$ for all $v \in I$, and let's denote the divisor $D =  \sum_{v \in I, 1 \leq i \leq r_v} a_{v, i} p_{v, i} -   \sum_{v \in I, 1 \leq i \leq n_v} b_{v, i} q_{v, i}$ with rational coefficients, and the rational line bundle ( which is in fact also an ordinary line bundle in this case) associated to it by $\calO_{Z}(D)$ where $\calO_{Z}(D) \in  \pic^{l'}(Z)$.

1)  With the notations above $h^1(\calO_{Z}(D)) \leq h^1(Z) - d_{Z, l'}$, where the upper bound is the $h^1$ of a generic line bundle in $\im(c^{l'}(Z))$.

2) If $I = I'$, then $\calO_{Z}(D) \in  \overline{ \im(c^{l'}(Z))}$ and $h^1(\calO_{Z}(D)) = h^1(Z) - d_{Z, l'}$.

3) If $I \neq I'$, then $\calO_{Z}(D) \notin  \im(c^{l'}(Z))$.
\end{lemma}

\begin{remark}
Notice, that if the analytic type of the integer cycle $Z$ is generic, then the numbers $d_{Z, l'}$ can be computed combinatorially just from the resolution graph by \cite{NNAD}.
Similary the number $ \dim(V_v(Z) \oplus \im(c^{l'}(Z)))$ can be computed combinatorially, because it equals to $ \dim(\im(c^{l' - N \cdot E_v^*}(Z)))$, where $N$ is a very large integer number. 

This means, that the property $I = I'$ can be tested  and the number $ h^1(Z) - d_{Z, l'}$ can be computed combinatorially from the resolution graph, if the analytic type of the singularity $\tX$ or the cycle $Z$ is generic.
\end{remark}

\begin{proof}

For (1) notice first, that the type of divisors like $D$ can be specilased to divisors where $n_v = 0$, $r_v = 1$, $a_{v,1} = a_v > 0$, $a_v \in \bZ$, but $p_v = p_{v, 1}$ is still generic on $E_v$.
Since $h^1$ is semicontinous we are enough to prove the statement in this special case.

So we have $D = \sum_{v \in I} a_v p_v \in \eca^{l'}(Z)$ and we would like to argue, that $h^1(\calO_{Z}(D)) = h^1(Z) - d_{Z, l'}$.

From \cite{NNA1} we know, that $h^1(\calO_Z) - h^1(\calO_{Z}(D)) =  \dim(\im(T_D(c^{l'}(Z)))) $ and if $D' \in \eca^{l'}(Z)$ is a generic divisor, then $h^1(\calO_Z) - h^1(\calO_{Z}(D')) =  \dim(\im(T_{D'}(c^{l'}(Z))))$.

 It means, that we have to prove, that $\dim(\im(T_D(c^{l'}(Z)))) = d_{Z, l'} $, which is equal to $\dim(\im(T_{D'}(c^{l'}(Z))))$, where $D' \in \eca^{l'}(Z)$ is a generic divisor.

Let's denote the vertices in $I$ by $v_1, \cdots v_{|I|}$, we prove the following statement by induction on the parameter $0 \leq i \leq |I|$:

Let $p_1 \in E_{v_1}, \cdots , p_i \in E_{v_i}$ be generic points, and $D'_j \in \eca^{-a_{v_j} E_{v_j}^*}(Z)$ generic divisors for $i+1 \leq j \leq |I|$, then $h^1(Z, \calO_Z(\sum_{1 \leq j \leq i}a_{v_j} p_j +  \sum_{i+1 \leq j \leq |I|}D'_j )) =  h^1(Z) - d_{Z, l'}$.

Now for $i= 0$ the statement is trivial, and for $i = |I|$ it yields our statement, so we have to do just the induction step.

So assume, that $i > 0$ and let's have generic points $p_1 \in E_{v_1}, \cdots , p_i \in E_{v_i}$ and $D'_j \in \eca^{-a_{v_j} E_{v_j}^*}(Z)$ generic divisors for $i+1 \leq j \leq |I|$, and let's denote $D_{I \setminus v_i} = \sum_{1 \leq j \leq i-1}a_{v_j} p_j + \sum_{i+1 \leq j \leq |I|}D'_j $.

We know by induction, that $\dim(\im (T_{D_{I \setminus v_i} + D'_i}(c^{l'}(Z)))) = d_{Z, l'}$ for a generic divisor $D'_i \in \eca^{-a_{v_i} E_{v_i}^*}(Z)$, and we want to prove that 
 $\dim(\im (T_{D_{I \setminus v_i} + a_{v_i} p_i}(c^{l'}(Z)) )) = d_{Z, l'}$ for a generic point $p_i \in E_{v_i}$.

Let's denote the linearization of $\im \left(T_{D_{I \setminus v_i}}(c^{l' +a_{v_i}E_{v_i}^*}(Z)) \right) $ by $V$ which is a subspace of $\pic^{l'}(Z) \cong H^1(\calO_Z)$, and let's denote the quotient map $H^1(\calO_Z) \to H^1(\calO_Z) / V$ by $\pi$.

Let's consider the map $ f: S^{a_{v_i}}(E_{v_i, reg}) \to \pic^{l'}(Z) / V$, where $E_{v_i, reg}$ is the smooth part of the exceptional divisor $E_{v_i}$ and $f(x)$ is the coset of the line bundle $c^{l'}(Z)( x + D_{I \setminus v_i})$.

We have $\im (T_{D_{I \setminus v_i} + D'_i}(c^{l'}(Z))) = \pi^{-1}(\im(T_{D'_i}f))$ and 
similarly $\im(T_{D_{I \setminus v_i} + a_{v_i} p_i}(c^{l'}(Z)))= \pi^{-1}(\im(T_{ a_{v_i} p_i}f))$.

However the preceeding lemma shows, that they have the same dimension, which proves part 1) of our lemma.

For part 2), assume first, that $I = I'$ and let's denote by $I'' \subset I$ the subset of vertices such that $ v \in I''$, if and only if  $n_v = 0$, and $a_{v, i} \in \bZ$ for all $1 \leq i \leq r_v$ and let's denote $I \setminus I''$ by $J$.

We know, that $\calO_{Z}(D) = \calO_{Z}( \sum_{v \in I'', 1 \leq i \leq r_v} a_{v, i} p_{v, i} + \sum_{v \in J, 1 \leq i \leq r_v} a_{v, i} p_{v, i} -   \sum_{v \in J, 1 \leq i \leq n_v} b_{v, i} q_{v, i})$.

For each vertex $v \in J$ let's choose $ (l', E_v)$ generic points $s_{v, 1}, \cdots, s_{v, (l', E_v)}$ and notice, that $ \calO_{Z}( \sum_{v \in I'', 1 \leq i \leq r_v} a_{v, i} p_{v, i} +  \sum_{v \in J, 1 \leq i \leq (l', E_v)}  s_{v, i}) \in  \im(c^{l'}(Z))$, let's denote this line bundle by $\calL$.

Notice that $\calO_{Z}(D) = \calL \otimes \calO_{Z}(  \sum_{v \in J, 1 \leq i \leq r_v} a_{v, i} p_{v, i} -   \sum_{v \in J, 1 \leq i \leq n_v} b_{v, i} q_{v, i} -   \sum_{v \in J, 1 \leq i \leq (l', E_v)}  s_{v, i}) \in \calL \oplus V_J(Z) \subset  \im(c^{l'}(Z)) \oplus V_J(Z)$.

We know, that $\dim(V_J(Z) \oplus \im(c^{l'}(Z))) = \dim(\im(c^{l'}(Z))$.

Indeed we know, that for each vertex $v \in J$ one has $V_v(Z) \oplus \im(c^{l'}(Z)) \subset \overline{\im(c^{l'}(Z))}$, since we have $V_J(Z) = \oplus_{v \in J} V_v(Z)$.

So it follows, that $V_J(Z) \oplus \im(c^{l'}(Z)) \subset \overline{\im(c^{l'}(Z))}$, so indeed we have $\dim(V_J(Z) \oplus \im(c^{l'}(Z))) = \dim(\im(c^{l'}(Z))$.

It means, that we have $V_J(Z) \oplus \im(c^{l'}(Z)) \subset \overline{ \im(c^{l'}(Z))}$, which yields, that $\calO_{Z}(D) \in \overline{ \im(c^{l'}(Z))}$.

Notice that by semicontinuity we get immediately, that $h^1(\calO_{Z}(D)) \geq h^1(Z) - d_{Z, l'}$, however by part 1) we know, that $h^1(\calO_{Z}(D)) \leq h^1(Z) - d_{Z, l'}$, which means that $h^1(\calO_{Z}(D)) = h^1(Z) - d_{Z, l'}$ and this proves part 2) completely.

For statement 3) we first prove a lemma, which will be cruical in the proof of our main statement:

\begin{lemma}
Let $\mathcal{T}$ be a resolution graph, and $\tX$ an arbitrary resolution of a normal surface singularity with resolution graph $\mathcal{T}$, $I \subset \calv$ an arbitrary subset and $Z \geq E$ an effective integer cycle, such that $Z_v = 1$ for all $v \in I$.

For all vertices $v \in I$ let's have integers $r_v > 0$, and rational numbers $a_{v,1}, \cdots , a_{v, r_v} > 0$, $a_v =  \sum_{1 \leq i \leq r_v} a_{v, i}$.

Let's have a rational line bundle $\calL$ on $\tX$ with Chern class $l'' \in L' \otimes \bQ$, such that $c_1(\calL) + \sum_{v \in \calv} a_v E_v^* \in L'$, let's denote this Chern class by $l' \in L'$.

Let's have generic ponts on $E_v$, $p_{v, 1}, \cdots p_{v, r_v}$ for all $v \in I$, and let's denote the rational divisor $D =  \sum_{v \in I, 1 \leq i \leq r_v} a_{v, i} p_{v, i}$ and assume that $H^0(Z, \calL \otimes \calO_Z(-D) )_{reg} \neq 0$.

For a vertex $v \in I$ and $1 \leq i \leq r_v$ let's denote $D(v, i) = 0$ if $a_{v, i} \in \bZ$ and $D(v, i) = 1$ otherwise.

For $v \in I$ and $1 \leq i \leq r_v$  let's have $[a_{v, i}] + D(v, i)$ generic points on the exceptional divisor $E_v$ and let's denote them by $s_{v, i, j}$, where $1 \leq j \leq [a_{v, i}] + D(v, i)$.

Furthermore let's denote $c_{v, i, j} = 1$ if $1 \leq j \leq [a_{v, i}]$ and if $D(v, i) = 1$, then $c_{v, i, j} = a_{v, i} - [a_{v, i}]$ if $j = [a_{v, i}] + 1$.

Let's have $D' = \sum_ {v \in I, 1 \leq i \leq r_v, 1 \leq j \leq [a_{v, i}] + D(v, i)} c_{v, i, j} \cdot s_{v, i, j}$, then we have $h^0(Z, \calL \otimes \calO_Z( - D)) = h^0(Z, \calL \otimes \calO_Z( - D'))$.
\end{lemma}
\begin{proof}

We can get the divisor $D$ by degeneration of $D'$, so by semicontinuity we immediately get, that $h^0(Z, \calL \otimes \calO_Z( - D)) \geq h^0(Z, \calL \otimes \calO_Z( - D'))$, so it is enough to prove in the following, that $h^0(Z, \calL \otimes \calO_Z( - D)) \leq h^0(Z, \calL \otimes \calO_Z( - D'))$.

So let's fix the rational line bundle $\calL$ and assume to the contrary that $h^0(Z, \calL \otimes \calO_Z( - D)) > h^0(Z, \calL \otimes \calO_Z( - D'))$.

Let's look at the counterexmaples, where the fractional parts of coeficcients of the points in the divisor are a subset of the fractional parts of the numbers $a_{v, i}$ .

We can assume, that $\sum_{v \in I, 1 \leq i \leq r_v} a_{v, i}$ is maximal, and we assume furthermore, that among these maximal cases $\sum_{v \in I, 1 \leq i \leq r_v} a_{v, i}^2$ is minimal among these counterexamples.

We know, that the values of  $\sum_{v \in I, 1 \leq i \leq r_v} a_{v, i}$ and  $\sum_{v \in I, 1 \leq i \leq r_v} a_{v, i}^2$ are bounded and can take only finitely many values because of
the condition on the fractional parts, so we can indeed make these assumptions.

Indeed we have $h^0(Z, \calL - D)_{reg} \neq 0$ which means, that $0 \leq \sum_{v \in I, 1 \leq i \leq r_v} a_{v, i} \leq (l'', E)$, from this the boundedness of $\sum_{v \in I, 1 \leq i \leq r_v} a_{v, i}^2$ also follows.

Now assume, that the sections in $h^0(Z, \calL \otimes \calO_Z( - D))$ have got a base point at some of the points $p_{v, i}$, where $v \in I$ and $1 \leq i \leq r_v$, then we have $h^0(Z, \calL \otimes \calO_Z( - D)) = h^0(Z, \calL \otimes \calO_Z( - D- p_{v, i}))$ and $H^0(Z, \calL \otimes \calO_Z( - D- p_{v, i}))_{reg} \neq 0$.

We know, that $ D + p_{v, i}$ cannot be a counterexample for the lemma, because $\sum_{v \in I, 1 \leq i \leq r_v} a_{v, i}$ was maximal among the counterexamples.

It means, that if we have a generic point $s' \in E_v$, then $h^0(Z, \calL \otimes \calO_Z( - D- p_{v, i})) = h^0(Z, \calL \otimes \calO_Z( - D'- s'))$, which means, that $ h^0(Z, \calL \otimes \calO_Z( - D')) \geq  h^0(Z, \calL \otimes \calO_Z( - D'- s')) =  h^0(Z, \calL \otimes \calO_Z( - D)) $ and we are done.

So we can assume in the following, that the sections in $H^0(Z, \calL \otimes \calO_Z( - D)) $ hasn't got a base point at any of the points $p_{v, i}$.

We know, that there exists a vertex $v \in I$ and $1 \leq i \leq r_v$, such that $a_{v, i} > 1$ otherwise we could take $D' = D$ and then trivially we get $ h^0(Z, \calL \otimes \calO_Z( - D)) =  h^0(Z, \calL \otimes \calO_Z( - D')) $.

So suppose that $u \in I$ and $1 \leq j \leq r_v$, such that $a_{u, j} > 1$ and let's have a section $s \in  h^0(Z, \calL \otimes \calO_Z( - D)) $ such that $|s| \cap |D| = \emptyset$
and let's denote $D'' = |s|$.

Let's look at the rational line bundle $\calL' = \calL \otimes \calO_Z(- D + a_{u, j} p_{u, j})$ and notice that $\calO_Z( a_{u, j} p_{u, j} + D'') = \calL' $.

We can fix the rational line bundle $\calL'$ and move $p_{u, j}$ as a generic point in $E_u$, while we have always an apporpriate section $s \in H^0(Z, \calL \otimes \calO_Z( - D)) $ corresponding to it with $D'' = |s|$.

Now, let's have the point $(D'', p_{u, j}) \in \eca^{l'}(Z) \otimes E_{u, reg}$, and let $T$ be the subspace of the tangent space of $ \eca^{l'}(Z) \otimes E_{u, reg}$ in the point $(D'', p_{u, j})$, which is the pullback of the tangent space of $ \eca^{l'}(Z)$ in $D''$.

Let's have the map $g : \eca^{l'}(Z) \otimes E_{u, reg} \to \pic^{c_1(\calL')}(Z)$ given by $ g(D^*, q) = \calO_{Z}(D^* + a_{u, j} q)$.

We know that in $g^{-1}( \calO_Z(D'' +   a_{u, j} p_{u, j}))$ the second coordinate is not constant, which yields that (for generic choice of the point $ p_{u, j}$) the kernel of the tangent map $T_{(D'', p_{u, j})}g : T_{(D'', p_{u, j})} \left(\eca^{l'}(Z) \otimes E_{u, reg} \right) \to T_{\calO_Z(D'' +   a_{u, j} p_{u, j} )} \pic^{c^1(\calL')}(Z)$ isn't contained in the subspace $T$.

This also means, that if we look at the map $g' : \eca^{l'}(Z) \otimes E_{u, reg} \to \pic^{l' - E_u^*}(Z)$ given by $ g'(D^* , q) = \calO_{Z}(D^* + q)$, then the kernel of the tangent map $T_{(D'', p_{u, j})}g' : T_{(D'', p_{u, j})} \left(\eca^{l'}(Z) \otimes E_{u, reg} \right) \to  T_{\calO_Z(D'' +  p_{u, j} )} \pic^{l' - E_u^*}(Z)$ isn't contained in the subspace $T$.

This means however by \cite{NNA1}, that $g'^{-1}(\calO_Z(D'' + p_{u, j}))$ has got nonconstant second coordinate, so if $q$ is a generic point on $E_u$, then $ \calO_Z(D'' + p_{u, j} - q) \in \im(c^{l'}(Z))$.

This means, that $h^0( \calO_Z(D'' + p_{u, j})) > h^0(\calO_Z( D''))$ and if $q$ is a generic point of $E_u$, then $h^0(\calO_Z( D'' + p_{u, j}-q)) = h^0(\calO_Z( D'' + p_{u, j})) - 1$, which means, that $h^0(\calO_Z(D'' + p_{u, j}-q)) \geq h^0(\calO_Z( D'')) = h^0(Z, \calL \otimes \calO_Z( - D))$.

Notice that $\calO_Z(D'' + p_{u, j}-q ) = \calL \otimes \calO_Z( - \sum_{v \in I, 1 \leq i \leq r_v, (v, i) \neq (u, j)} a_{v, i} p_{v, i} - (a_{u, j}-1) p_{u, j} - q) $ and notice, that  $\sum_{v \in I, 1 \leq i \leq r_v} a_{v, i}^2 > \sum_{v \in I, 1 \leq i \leq r_v, (v, i) \neq (u, j)} a_{v, i}^2 +  (a_{u, j}-1)^2 + 1$,
which means by the minimality of our counterexample, that $h^0(\calO_Z(D'' + p_{u, j}-q) ) \leq  h^0(Z, \calL \otimes \calO_Z( - D'))$, which means, that $h^0(Z, \calL \otimes \calO_Z( - D)) \leq h^0(Z, \calL \otimes \calO_Z( - D'))$, however this is a contradiction which proves our lemma completely.

\end{proof}

Now for statement 3) assume to the contrary, that $ \calO_Z(D) \in \im(c^{l'}(Z))$ and $I \neq I'$, which means, that there is a vertex $v \in I$, such that $\dim(V_v(Z) \oplus \im(c^{l'}(Z))) > \dim(\im(c^{l'}(Z))$ and we have $n_v > 0$ or there is an index $1 \leq i \leq r_v$ such that $a_{v, i} \notin \bZ$.

First we want to argue, that we can assume, that $Z$ is the cohomological cycle of a generic line bundle in $ \im(c^{l'}(Z))$.

Indeed assume, that we know the statement in this case and assume now, that the cohomological cycle of a generic line bundle in $ \im(c^{l'}(Z))$ is some cycle $0 \leq Z' < Z$.

By \cite{NNAD} this means, that $ \im(c^{l'}(Z))$ is birational to a fibration over $ \im(c^{l'}(Z'))$ where the fibres are complex vector spaces of dimension $h^1(\calO_Z) - h^1(\calO_Z')$.

We know, that there is a vertex $v \in I$ such that $\dim(V_v(Z) \oplus \im(c^{l'}(Z))) > \dim(\im(c^{l'}(Z)))$ and we have $n_v > 0$ or there is an index $1 \leq i \leq r_v$ such
that $a_{v, i} \notin \bZ$.

Now we get that $\dim(V_v(Z') \oplus \im(c^{l'}(Z'))) > \dim(\im(c^{l'}(Z'))$:

Indeed, let's denote the kernel of the linear surjection $\pi : H^1(\calO_Z) \to H^1(\calO_Z)$ by $K$.

We know, that $ \overline{K \oplus \im(c^{l'}(Z))} = \overline{\im(c^{l'}(Z)}$, $\pi(\im(c^{l'}(Z)) = \im(c^{l'}(Z'))$ and $\pi(V_v(Z)) = V_v(Z')$.

We get from these facts, that $\dim(V_v(Z') \oplus \im(c^{l'}(Z'))) +  h^1(\calO_Z) - h^1(\calO_Z')= \dim(V_v(Z) \oplus \im(c^{l'}(Z)))$ and
$\dim(\im(c^{l'}(Z'))) + h^1(\calO_Z) - h^1(\calO_Z') = \dim(\im(c^{l'}(Z)))$, and we indeed get, that  $\dim(V_v(Z') \oplus \im(c^{l'}(Z'))) > \dim(\im(c^{l'}(Z'))$.

In particular we get, that $E_v \leq Z'$.

We also know, that $n_v > 0$ or there is an index $1 \leq i \leq r_v$, such
that $a_{v, i} \notin \bZ$, so it means by the special case of the statement that $ \calO_{Z'}(D) \notin \im(c^{l'}(Z'))$, which indeed yields $ \calO_{Z}(D) \notin \im(c^{l'}(Z))$, since
$\calO_{Z'}(D) = \calO_{Z}(D) | Z'$.

It means, that in the following we can assume, that $Z$ is the cohomological cycle of a generic line bundle in $ \im(c^{l'}(Z))$.

We want to conlude from this, that $Z$ is the cohomological of the line bundle $\calO_Z(D)$, so that $h^1(Z', \calO_{Z'}(D)) < h^1(Z, \calO_{Z}(D))$ for every cycle $0 \leq Z' < Z$.

Indeed let's have a generic line bundle $\calL_{gen} \in \im(c^{l'}(Z))$, then we have $h^1(Z', \calL_{gen} | Z') < h^1(Z,  \calL_{gen})$ for every cycle $0 \leq Z' < Z$.

Let's choose generic points $s_v \in E_v$ for each vertex $v \in I$ and let's have the divisor $D^* = \sum_{v \in I} (a_v - b_v) s_v$, now by part 1) we know, that for every cycle $0 \leq Z' < Z$ we have $h^1(Z', \calO_{Z'}(D^*)) = h^1(Z',  \calL_{gen} | Z')$, which means, that $h^1(Z', \calO_{Z'}(D^*)) < h^1(Z, \calO_{Z}(D))$ for every cycle $0 \leq Z' < Z$, because we get $ h^1(Z, \calO_{Z}(D)) =  h^1(Z, \calO_{Z}(D^*))$ from the assumption $ \calO_Z(D) \in \im(c^{l'}(Z))$.

On the other hand by semicontinuity we have $h^1(Z', \calO_{Z'}(D)) \leq h^1(Z', \calO_{Z'}(D^*)) $, so indeed we get $h^1(Z', \calO_{Z'}(D)) < h^1(Z, \calO_{Z}(D))$ for every cycle $0 \leq Z' < Z$, which means, that  $Z$ is the cohomological cycle of the line bundle $\calO_Z(D)$.

Assume furthermore, that while the numbers $r_v$,  $a_{v,1}, \cdots , a_{v, r_v} > 0$ are fixed, this counterexample is extremal in the sense, that $\sum_{v \in I, 1 \leq i \leq n_v} b_{v, i}^2$ is minimal among the counterexamples $D' =  \sum_{v \in I, 1 \leq i \leq r_v} a'_{v, i} p_{v, i} -   \sum_{v \in I, 1 \leq i \leq n_v} b'_{v, i} q_{v, i}$, where the fractional parts of the numbers $b'_{v, i}$ are a subset of the fractional parts of the numbers $b_{v, i}$.

First we want to argue, that  $ b_{v, i} \leq 1$ for all $ v \in I$ and $1 \leq i \leq n_v$, so assume to the contrary and by symmetry, that $b_{w, k} > 1$ for some $w \in I, 1 \leq k \leq n_w$.

Let's denote the rational line bundle $\calL = \calO_Z(\sum_{v \in I, 1 \leq i \leq r_v} a_{v, i} p_{v, i})$,  and notice that by the assumption if $q_{v, 1}, \cdots q_{v, n_v}$ are generic points on  $E_v$ for all  $v \in I$, then $H^0(Z, \calL \otimes \calO_Z(-  \sum_{v \in I, 1 \leq i \leq n_v} b_{v, i} q_{v, i}))_{reg} \neq \emptyset$.

For some vertex $v \in I$ and integer $1 \leq i \leq n_v$ let's denote $D(v, i) = 0$ if $b_{v, i} \in \bZ$ and $D(v, i) = 1$ otherwise.

For $v \in I$ and $1 \leq i \leq n_v$  let's have $[b_{v, i}] + D(v, i)$ general points on $E_v$ and let's denote them by $s_{v, i, j}$, where $1 \leq j \leq [b_{v, i}] + D(v, i)$.
Furthermore let's denote $c_{v, i, j} = 1$ if $1 \leq j \leq [b_{v, i}] $ and if $D(v, i) = 1$, then $c_{v, i, j} = b_{v, i} - [b_{v, i}]$ if $j = [b_{v, i}] + 1$.

Let's have $D' = \sum_ {v \in I, 1 \leq i \leq n_v, 1 \leq j \leq [b_{v, i}] + D(v, i)} c_{v, i, j} \cdot s_{v, i, j}$, then by the previous lemma we have $h^0(Z, \calL\otimes \calO_Z( -  \sum_{v \in I, 1 \leq i \leq n_v} b_{v, i} q_{v, i})) = h^0(Z, \calL \otimes \calO_Z( - D'))$.

Notice, that $H^0(Z, \calL \otimes \calO_Z(-  \sum_{v \in I, 1 \leq i \leq n_v} b_{v, i} q_{v, i}))_{reg} \neq \emptyset$ which means, that:

\begin{equation*}
h^0(Z, \calL \otimes \calO_Z(-  \sum_{v \in I, 1 \leq i \leq n_v} b_{v, i} q_{v, i})) > h^0(Z-A, \calL| (Z-A) \otimes \calO_{Z-A}(-  \sum_{v \in I, 1 \leq i \leq n_v} b_{v, i} q_{v, i} - A)),
\end{equation*}
 for every cycle $0 < A \leq Z$.

On the other hand by semicontinuity we have:

\begin{equation*}
 h^0 ( Z- A, \calL| (Z-A) \otimes  \calO_{Z-A}(-  \sum_{v \in I, 1 \leq i \leq n_v} b_{v, i} q_{v, i} - A) ) \geq h^0 ( Z - A, \calL|(Z-A) \otimes \calO_{Z-A}( - D'-A)),
\end{equation*}
 for every cycle $0 \leq A \leq Z$.

This means, that $h^0(Z , \calL \otimes \calO_Z(- D')) > h^0(Z - A, \calL|(Z-A) \otimes \calO_{Z-A}( - D'-A))$ for every cycle $0 < A \leq Z$, which means, that $H^0(Z , \calL \otimes \calO_Z(- D'))_{reg} \neq \emptyset$.

It means, that $\calL \otimes \calO_Z(- D') \in \im(c^{l'}(Z))$, however by the minimality of $\sum_{v \in I, 1 \leq i \leq n_v} b_{v, i}^2$
among the counterexamples $ \sum_{v \in I, 1 \leq i \leq r_v} a'_{v, i} p_{v, i} -   \sum_{v \in I, 1 \leq i \leq n_v} b'_{v, i} q_{v, i}$, where the fractional parts of the numbers $b'_{v, i}$ are a subset of the fractional parts of the numbers $b_{v, i}$ we have a contradiction.

This means, that in the following we can assume, that  $ b_{v, i} \leq 1$ for all $ v \in I$ and $1 \leq i \leq n_v$.

In the following fix the numbers $n_v$,  $b_{v,1}, \cdots , b_{v, n_v} > 0$ and assume that this counterexample is extremal in the sense, that $\sum_{v \in I, 1 \leq i \leq r_v} a_{v, i}^2$ is minimal among the counterexamples, where the fractional parts of the numbers $a'_{v, i}$ are a subset of the fractional parts of the numbers $a_{v, i}$.

We want to argue in the following, that $a_{v, i} \leq 1$ for every $v \in I$, $1 \leq i \leq r_v$, indeed assume to the contrary, that $a_{w, k} > 1$ for some $w \in I, 1 \leq k \leq r_w$.

By Seere duality, we have $h^1(Z, \calO_Z(D)) = h^0(Z, \calO_Z(K + Z - D))$, and we know, that $Z$ is the cohomological cycle of the line bundle $\calO_Z(D)$, so we get, that $ H^0(Z, \calO_Z(K + Z - D))_{reg} \neq \emptyset$.

Let's denote the rational line bundle $\calO_Z(K + Z + \sum_{v \in I, 1 \leq i \leq n_v} b_{v, i} q_{v, i})$ by $\calL$, notice that by the assumption if $p_{v, 1}, \cdots p_{v, r_v}$ are generic points on  $E_v$ for all  $v \in I$, then $H^0(Z, \calL  \otimes \calO_Z(-  \sum_{v \in I, 1 \leq i \leq r_v} a_{v, i} p_{v, i}))_{reg} \neq \emptyset$.

Now for some vertex $v \in I$ and $1 \leq i \leq r_v$ let's denote $G(v, i) = 0$ if $a_{v, i} \in \bZ$ and $G(v, i) = 1$ otherwise.

For $v \in I$ and $1 \leq i \leq r_v$  let's have $[a_{v, i}] + G(v, i)$ general points on $E_v$ and let's denote them by $s_{v, i, j}$, where $1 \leq j \leq [a_{v, i}] + G(v, i)$.
Furthermore let's denote $c_{v, i, j} = 1$ if $1 \leq j \leq [a_{v, i}] $ and if $G(v, i) = 1$, then $c_{v, i, j} = a_{v, i} - [a_{v, i}]$ if $j = [a_{v, i}] + 1$.

Let's have $D' = \sum_ {v \in I, 1 \leq i \leq r_v, 1 \leq j \leq [a_{v, i}] + G(v, i)} c_{v, i, j} \cdot s_{v, i, j}$, then by the previous lemma we have $h^0(Z, \calL \otimes \calO_Z( -  \sum_{v \in I, 1 \leq i \leq r_v} a_{v, i} p_{v, i})) = h^0(Z, \calL  \otimes \calO_Z(- D'))$.

This means by Seere duality, that $h^0(Z, \calO_Z(D)) = h^0(Z, \calO_Z(D' -  \sum_{v \in I, 1 \leq i \leq n_v} b_{v, i} q_{v, i}))$.

Notice, that by semicontinuity for every cycle $0 < A \leq Z$ we have $h^0(Z- A, \calO_{Z-A}(D-A)) \geq h^0(Z-A, \calO_{Z-A}( D' -  \sum_{v \in I, 1 \leq i \leq n_v} b_{v, i} q_{v, i} - A))$ and we have $h^0(Z, \calO_Z(D)) > h^0(Z- A, \calO_{Z-A}(D-A))$, which means, that $ H^0(Z, \calO_Z(D' -  \sum_{v \in I, 1 \leq i \leq n_v} b_{v, i} q_{v, i}))_{reg} \neq \emptyset$.

This contradicts the minimality of the value $\sum_{v \in I, 1 \leq i \leq r_v} a_{v, i}^2$, which means, that in the following we can assume, that $a_{v, i} \leq 1$ for all $v \in I, 1 \leq i \leq r_v$ and $b_{v, i} \leq 1$ for all $v \in I, 1 \leq i \leq n_v$.

Notice, that since $H^0(Z, \calO_Z(D))_{reg} \neq \emptyset $ for general points $p_{v, i}, q_{v, i}$ we have $\oplus_{v \in I, 1 \leq i \leq r_v}  a_{v, i} \cdot \im(c^{-E_v^*}(Z)) \oplus \oplus_{v \in I, 1 \leq i \leq n_v}  -b_{v, i} \cdot \im(c^{-E_v^*}(Z)) \subset \overline{\im(c^{l'}(Z))}$,
let's denote $M = \oplus_{v \in I, 1 \leq i \leq r_v}  a_{v, i} \cdot \im(c^{-E_v^*}(Z)) \oplus \oplus_{v \in I, 1 \leq i \leq n_v}  -b_{v, i} \cdot \im(c^{-E_v^*}(Z))$.

Let's denote in the following $l'_1 = \sum_{v \in I}- (r_v + n_v) E_v^*$, we claim, that $\dim(M) = \dim(\im(c^{l'_1}(Z)))$.

Indeed, notice that $\im(c^{l'_1}(Z)) = \oplus_{v \in I, 1 \leq i \leq r_v}  \im(c^{-E_v^*}(Z)) \oplus \oplus_{v \in I, 1 \leq i \leq n_v} \im(c^{-E_v^*}(Z))$.

Since the Minkowski map $f:  \otimes_{v \in I, 1 \leq i \leq r_v}  \im(c^{-E_v^*}(Z)) \otimes \otimes_{v \in I, 1 \leq i \leq n_v} \im(c^{-E_v^*}(Z)) \to \im(c^{l'_1}(Z))$ is dominant, we know, that if $z_{v, i}$ are generic points in $\im(c^{-E_v^*}(Z))$ for $v \in I, 1 \leq i \leq r_v$
and $t_{v, i}$ are generic points in $\im(c^{-E_v^*}(Z))$ for $v \in I, 1 \leq i \leq n_v$, then:

\begin{equation*}
\dim(\im(c^{l'_1}(Z))) = \dim \left( \oplus_{v \in I, 1 \leq i \leq r_v} T_{z_{v, i}}\left( \im(c^{-E_v^*}(Z))\right) \oplus \oplus_{v \in I, 1 \leq i \leq n_v} T_{t_{v, i}}\left( \im(c^{-E_v^*}(Z) ) \right) \right).
\end{equation*}

Similarly the Minkowski map $g: \otimes_{v \in I, 1 \leq i \leq r_v}  a_{v, i} \cdot \im(c^{-E_v^*}(Z)) \otimes \otimes_{v \in I, 1 \leq i \leq n_v} -b_{v, i} \cdot \im(c^{-E_v^*}(Z)) \to M$ is dominant, so we know, that:

\begin{equation*}
\dim(M) = \dim \left( \oplus_{v \in I, 1 \leq i \leq r_v} T_{a_{v, i} \cdot z_{v, i}}\left(  a_{v, i} \cdot \im(c^{-E_v^*}(Z))\right) \oplus \oplus_{v \in I, 1 \leq i \leq n_v} T_{- b_{v, i} \cdot t_{v, i}}\left(  -b_{v, i} \cdot  \im(c^{-E_v^*}(Z)) \right) \right).
\end{equation*}

Since $T_{z_{v, i}}\left( \im(c^{-E_v^*}(Z))\right) = T_{a_{v, i} \cdot z_{v, i}}\left(  a_{v, i} \cdot \im(c^{-E_v^*}(Z))\right)$ and $T_{ t_{v, i}}\left( \im(c^{-E_v^*}(Z)) \right)  = T_{- b_{v, i} \cdot t_{v, i}}\left(  -b_{v, i} \cdot  \im(c^{-E_v^*}(Z)) \right)$ we indeed get $\dim(M) = \dim(\im(c^{l'_1}(Z)))$.

It means, that we have $\dim(\im(c^{l'_1}(Z))) = \dim(M)  \leq \dim(\im(c^{l'}(Z)))$.

Notice, that since $a_{v, i} \leq 1$ for all $v \in I, 1 \leq i \leq r_v$ and $b_{v, i} \leq 1$ for all $v \in I, 1 \leq i \leq n_v$ we have $(l'_1, E_v) \geq (l', E_v)$ for every vertex $v \in |Z|$.

Let's recall also, that there is a vertex $u \in I$, such that $\dim(V_u(Z) \oplus \im(c^{l'}(Z))) > \dim(\im(c^{l'}(Z))$ and we have $n_u > 0$ or there is an index $1 \leq i \leq r_u$, such
that $a_{u, i} \notin \bZ$.

Notice, that we have $(l', E_u)< r_u + n_u = (l'_1, E_u)$, so $\dim(\im(c^{l'_1}(Z))) \geq \dim(\im(c^{-E_{u}^*}(Z)) \oplus \im(c^{l'}(Z)))$.

We know, that $\dim(V_u(Z) \oplus \im(c^{l'}(Z))) > \dim(\im(c^{l'}(Z)))$ from which it follows, that $\dim(\im(c^{-E_{u}^*}(Z)) \oplus \im(c^{l'}(Z))) > \dim(\im(c^{l'}(Z))$.

It yields $\dim(M) > \dim(\im(c^{l'}(Z)))$, which is a contradicition and it proves part 3) of our main lemma completely.

\end{proof}

Now we are ready to prove our main theroem of this article:

\begin{theorem}
Let $\mathcal{T}$ be an arbitrary resolution graph, and let's have a generic singularity and resolution $\tX$ corresponging to the resolution graph $\mathcal{T}$.
Let's have furthermore an effective integer cycle $Z \in L$, and an arbitrary Chern class $Z' \in L'$, 
then the cohomology number $h^1(\calO_Z(Z'))$ can be computed from the resolution graph and $Z, Z'$ combinatorially.
\end{theorem}
\begin{proof}

We will prove the theorem by induction on $h^1(\calO_Z)$ (which is a combinatorially determined number by the main theorem of \cite{NNA2}, namely $h^1(\calO_Z) = \chi(E_{|Z|})  - \min_{E_{|Z|} \leq l \leq Z} \chi(l)$).

Of course, if $h^1(\calO_Z) = 0$, then a line bundle on $Z$ is determined by its Chern class, and in this case $h^1(\calO_Z(Z')) = \chi(Z') - \min_{0 \leq l \leq Z} \chi(Z' + l)$,
which is equal to the $h^1$ of the generic line bundle in $\pic^{c_1(\calO_Z(Z'))}(Z)$.

Now assume, that the statement is proven for all cases $h^1(\calO_Z) \leq r-1$ and let's have $h^1(\calO_Z) = r$, we are proving the induction step in a several number of steps.

Step 1)  Assume first, that, $h^1(\calO_Z) = r$, and $|Z| \cap |Z'|= 0$. Furthermore assume, that for every $v \in |Z|$, such that, there exists a vertex $w \in |Z'|$, for which $(v, w)$ is an edge one has $Z_v = 1$.
Assume furthermore, that there isn't a vertex $v \in |Z|$, such that $h^1(Z) = h^1(\calO_{Z - Z_v \cdot E_v})$. With these properties we claim, that $h^1(\calO_Z(Z'))$ is combinatorially computable:

If $|Z|$ is nonconnencted and the connected components of $|Z|$ are $|Z_1|, \cdots |Z_i|$, then if $h^1(\calO_{Z_j}) < h^1(\calO_Z)$ for all $1 \leq j \leq i$, then the statement follows from our induction hypothesis, indeed we have $h^1(\calO_Z(Z')) = \sum_{1 \leq j \leq i} h^1(\calO_{Z_j}(Z'))$.

If $h^1(\calO_{Z_1}) = h^1(\calO_Z)$, and $h^1(\calO_{Z_j}) = 0$ for all $2 \leq j \leq i$, then we have:
\begin{equation*}
h^1(\calO_Z(Z')) = \sum_{2 \leq j \leq i}  \left( \chi(Z') - \min_{0 \leq l \leq Z_j} \chi(Z' + l) \right)  + h^1(\calO_{Z_1}(Z')).
\end{equation*}

So it means, that we just have to prove claim 1) in the case, when $|Z|$ is connected.

It means, that we are in the situation of the previous lemma, because the line bundle $\calO_Z(Z')$ on $Z$ can be written by $\calO_Z(D)$, where $D = \sum_{v \in I, 1 \leq i \leq r_v} a_{v, i} p_{v, i} -   \sum_{v \in I, 1 \leq i \leq n_v} b_{v, i} q_{v, i}$
for some rational numbers $a_{v, i}, b_{v, i}$ and generic points $p_{v, i}, q_{v, i} \in E_v , v\in I$, where $I$ is the set of vertices in $|Z|$ wich has got a neighbour in $|Z'|$.

Let's have the subset $I' \subset I$ consisting of vertices $v \in I$, such that $ \dim(V_v(Z) \oplus (\im(c^{l'}(Z))) = \dim(\im(c^{l'}(Z)))$ or we have $n_v = 0$ and $a_{v, i} \in \bZ$ for all $1 \leq i \leq r_v$.

Now we know from part 1) of our main lemma, that if $l' = c_1(\calO_Z(D)) \in  -S'$ and $I = I'$, then $\calO_Z(Z') \in \overline{ \im(c^{l'}(Z))}$ and $h^1(\calO_Z(Z')) = h^1(Z) - d_{Z, l'}$.

Notice, that the dimensions $ d_{Z, l'}$ are combinatorially computable from the resolution graph in the generic case by \cite{NNAD}, so in this case we are done.

On the other hand, if $l' \notin  -S'$ or $I \neq I'$, then we know from part 2) of our main lemma, that $\calO_Z(Z') \notin  \im(c^{l'}(Z))$, which means, that $h^0(\calO_Z(Z')) = \max_{0 <l \leq Z} h^0(\calO_{Z-l}(Z'-l))$, and furthermore, 
there is a cycle $0 < l \leq Z$, such that $h^0(\calO_Z(Z')) = h^0(\calO_{Z-l}(Z'-l))$ and $H^0(\calO_{Z-l}(Z'-l))_{reg} \neq \emptyset$ or $l= Z$.

For each  $0 < l \leq Z$ let's choose a vertex $ v_l \in |l|$, and let's denote $(Z-l)_{|Z-l| \setminus v_l}$ by $Z_l$, notice that $h^1(Z_l) \leq r-1$ surely by our assumption in case 1).

If $H^0(\calO_{Z-l}(Z'-l))_{reg} \neq \emptyset$ for some $0 < l < Z$, then using Theorem \ref{relgen} in the case $\calv_1 = |Z_l|$ , $\calv_2 = v_l$, $Z_1 = Z_l$ and $Z_2 = Z-l - Z_l$, we get $h^1(\calO_{Z-l}(Z'-l)) = h^1(\calO_{Z_l}(Z'-l))$.

We have obviously $h^1(\calO_{Z-l}(Z'-l)) = h^1(\calO_{Z_l}(Z'-l))$ also in the case $l = Z$.

On the other hand, if $H^0(\calO_{Z-l}(Z'-l))_{reg} = \emptyset$ for some $0 < l < Z$, we still have $h^1(\calO_{Z-l}(Z'-l)) \geq h^1(\calO_{Z_l}(Z'-l))$.

Notice, that by induction all $h^1(\calO_{Z_l}(Z'-l))$ is combinatorially computable and $h^0(\calO_Z(Z')) = \max_{0 <l  \leq Z} (h^1(\calO_{Z_l}(Z'-l)) + \chi(\calO_{Z-l}(Z'-l)))$ which proves claim 1).

Step 2) Assume, that, $h^1(\calO_Z) = r$, and $|Z| \cap |Z'|= 0$. Furthermore assume, that if we denote the set of vertices $v \in |Z|$ by $I$, for which, there exists a vertex $w \in |Z'|$ such that $(v, w)$ is an edge, then one has $ h^1(\calO_{Z_{|Z| \setminus I} + E_I}) = r$.  Assume furthermore, that there isn't a vertex
$v \in |Z|$, such that $h^1(\calO_Z) = h^1(\calO_{Z - Z_v \cdot E_v})$. With these properties $h^1(\calO_Z(Z'))$ is combinatorially computable:

We can again assume $|Z|$ is connected.

Notice that by Step 1), $h^1(\calO_{Z_{|Z| \setminus I} + E_I}(Z'))$ is combinatorially computable from the resolution graph.

Notice first, that $h^1(\calO_Z(Z')) \geq h^1(\calO_{Z_{|Z| \setminus I} + E_I}(Z'))$ and if $H^0(\calO_Z(Z'))_{reg} \neq \emptyset$, then equality happens.

Indeed from \cite{NNA1} we know, that if $B \leq A$ are two integer effective cycles on a surface singularity, such that $h^1(\calO_B) = h^1(\calO_A)$ and $\calL \in \pic^{l'}(A)$
is a line bundle, such that $\calL \in \im(c^{l'}(A))$, then we have $h^1(A, \calL) = h^1(B, \calL | B)$.

On the other hand if $H^0(\calO_Z(Z'))_{reg} = \emptyset$, then $h^0(\calO_Z(Z')) = \max_{0 <l \leq Z} h^0(\calO_{Z-l}(Z'-l))$, and furthermore, there is a cycle $0 < l \leq Z$, such that $h^0(\calO_Z(Z')) = h^0(\calO_{Z-l}(Z'-l))$ and $H^0(\calO_{Z-l}(Z'-l))_{reg}$ is nonempty or $l = Z$.

Obviously $h^0(\calO_Z(Z')) \geq \max_{0 <l \leq Z} h^0(\calO_{Z-l}(Z'-l))$ always happens.

Now for each  $0 < l \leq Z$, let's choose a vertex $ v_l \in |l|$, and let's denote $(Z-l)_{|Z-l| \setminus v_l}$ by $Z_l$.

If $H^0(\calO_{Z-l}(Z'-l))_{reg} \neq \emptyset$ for some $0 < l \leq Z$ or $l = Z$, then using Theorem \ref{relgen} in the case $\calv_1 = |Z_l|$ , $\calv_2 = v_l$, $Z_1 = Z_l$ and $Z_2 = Z-l - Z_l$, we get $h^1(\calO_{Z-l}(Z'-l)) = h^1(\calO_{Z_l}(Z'-l))$.

On the other hand if $H^0(\calO_{Z-l}(Z'-l))_{reg} = \emptyset$ we also have $h^1(\calO_{Z-l}(Z'-l)) \geq h^1(\calO_{Z_l}(Z'-l))$.

This means, that $h^0(\calO_Z(Z')) = \max(\chi(\calO_Z(Z')) +  h^1(\calO_{Z_{|Z| \setminus I} + E_I}(Z')), \max_{0 < l \leq Z }h^1(\calO_{Z_l}(Z'-l)) + \chi(\calO_{Z-l}(Z'-l)))$ and we have proved claim 2) with it.

Step 3)  Assume, that, $h^1(Z) = r$, and $|Z| \cap |Z'|= 0$. Assume furthermore, that there isn't a vertex $v \in |Z|$, such that $h^1(Z) = h^1(Z - Z_v \cdot E_v)$. With these properties $h^1(\calO_Z(Z'))$ is combinatorially computable:

Similarly as in the previous cases we can assume $|Z|$ is connected.

Let's denote the set of vertices $v \in |Z|$, for which, there exists a vertex $w \in |Z'|$ such that $(v, w)$ is an edge by $I$.

For a vertex $v \in I$ let's blow up $E_v$ sequentially in generic points, let the new exceptional divisors be $E_{v, 1}$, $E_{v,2}, \cdots E_{v, i}$, and let's have the cycle $Z_{v, i} = Z + \sum_{1 \leq j \leq i} Z_v E_{v, j}$ on the $i$-th blowup, and let $t_v$
be the minimal number, such that $h^1(\calO_{Z_{v, t_v}}) = h^1(\calO_Z) = h^1(\calO_{Z_{v, t_v} - Z_v \cdot E_{v, t_v}})$, we know from our conditions that $t_v \geq 1$.

We prove the statement by induction on the value of $\left( \sum_{v \in I}(Z', E_v) \cdot t_v \right)$.

If we have $t_v = 1$ for all $v \in I$, then we have $ h^1(\calO_{Z_{|Z| \setminus I} + E_I}) = r$.

Indeed, by our assumption every differential form in $\frac{H^0(\calO_{\tX}(K + Z))}{H^0(\calO_{\tX}(K))}$ must have a pole of order at most $1$ along the exceptional divisors $E_v, v \in I$.
The reason of it is because if $\omega \in \frac{H^0(\calO_{\tX}(K + Z))}{H^0(\calO_{\tX}(K))}$ has got a pole along an exceptional divisor $E_v, v \in I$ of order greater than $1$, then if we blow up $E_v$ at a generic point, then $\omega$ also has got a pole along the new exceptional divisor $E_{v_1}$, which means by \cite{NNA1}, that $h^1(\calO_{Z_{v, 1}}) = h^1(\calO_Z) > h^1(\calO_{Z_{v, 1} - Z_v \cdot E_{v, 1}})$, however this is a contradiction.

This means, that $h^1(\calO_{Z_{|Z| \setminus I} + E_I}) = h^1(\calO_Z) = r$.

In this case the statement follows from our previous case.

Now assume, that $\sum_{v \in I} (Z', E_v) \cdot t_v = t$ and we know the statement for $\sum_{v \in I} (Z', E_v) \cdot t_v < t$, and furthermore assume, that $t_v > 1$ for some vertex $v \in I$.

We know, that there is a vertex $w \in |Z'|$, such that $(v, w)$ is an edge, let's have the intersection point $E_v \cap E_w = p$.

Let's blow up $p$, and let's denote the blow down map by $\pi$.

We get a generic resolution of the new graph, let's denote the new exceptional divisor by $E_{v_1}$, and let's have $Z_{v, new} = Z + Z_v \cdot E_{v_1}$.

From the condition on $v$ and from the condition of our statement we know, that there isn't a vertex $u \in |Z_{v, new}|$, such that $h^1(\calO_{Z_{v, new}}) = h^1(\calO_{Z_{v, new} - Z_u \cdot E_u})$.

Indeed for $u \neq v_1$ this follows from the conditions of our statement, and for $u = v_1$ this follows, from $t_v > 1$.

We know, that $h^1(\calO_Z(Z')) = h^1(\calO_{Z_{v, new}}(\pi^*(Z')))$, so it is enough to compute $ h^1(\calO_{Z_{v, new}}(\pi^*(Z')))$.

We know, that $ h^0(\calO_{Z_{v, new}}(\pi^*(Z'))) =  \max_{0 \leq l \leq Z_{v, new}} h^0(\calO_{Z_{v, new} -l}(\pi^*(Z')-l))$, and furthermore, 
there is a cycle $0 \leq l \leq Z_{v, new}$, such that $h^0(\calO_{Z_{v, new}}(\pi^*(Z'))) = h^0(\calO_{Z_{v, new} -l}(\pi^*(Z')-l))$ and $H^0(\calO_{Z_{v, new} -l}(\pi^*(Z')-l))_{reg} \neq \emptyset$ or $l = Z_{v, new}$.

Now if $l \neq Z'_w \cdot E_{v,1}$, then there is a vertex $v_l \in |Z_{v, new}| \cap |\pi^*(Z')-l|$, namely $v_l = v_1$, and let's denote $(Z_{v, new}-l)_{|Z_{v, new} -l| \setminus v_l}$ by $Z_l$, if $l = Z_{v, new}$, then let's denote $Z_l = 0$, notice that in any case we have $h^1(Z_l) \leq r-1$.

If $H^0(\calO_{Z_{v, new} -l}(\pi^*(Z')-l))_{reg} \neq \emptyset$ for some $0 \leq l \leq Z_{v, new}$ and $l \neq Z'_w \cdot E_{v,1}$, then using Theorem \ref{relgen} in the case $\calv_1 = |Z_l|$ , $\calv_2 = v_l$, $Z_1 = Z_l$ and $Z_2 = Z_{v, new} -l - Z_l$, we get, that $h^1(\calO_{Z_{v, new} -l}(\pi^*(Z') -l)) = h^1(\calO_{Z_l}(\pi^*(Z') -l))$.

Let's denote $N =h^0(\calO_{Z_{v, new} - Z'_w \cdot E_{v,1}}(\pi^*(Z') - Z'_w \cdot E_{v,1})$, this means that:
\begin{equation*}
h^0(\calO_{Z_{v, new}}(\pi^*(Z'))) = \max \left( \max_{0 \leq l \leq Z_{v, new}, l \neq Z'_w \cdot E_{v,1}} \left( h^1(\calO_{Z_l}(\pi^*(Z')-l)) + \chi(\calO_{Z_{v, new} -l}(\pi^*(Z') -l))\right), N \right)
\end{equation*}

The term $\max_{0 \leq l \leq Z_{v, new}, l \neq Z'_w \cdot E_{v,1}} \left( h^1(\calO_{Z_l}(\pi^*(Z')-l)) + \chi(\calO_{Z_{v, new} -l}(\pi^*(Z') -l))\right)$ is computable because of $h^1(\calO_{Z_l}) \leq r-1$ for every $0 \leq l \leq Z_{v, new}, l \neq Z'_w \cdot E_{v,1}$, so we only need to show that $N = h^0(\calO_{Z_{v, new} - Z'_w \cdot E_{v,1}}(\pi^*(Z') - Z'_w \cdot E_{v,1})$ is also computable.

However notice, that $h^0(\calO_{Z_{v, new} - Z'_w \cdot E_{v,1}}(\pi^*(Z') - Z'_w \cdot E_{v,1})$ satisfies the conditions of our claim, and $\sum_{v \in I} (Z', E_v) \cdot t_v$ decreased.

Indeed the vertex $w$ is now the neighbour of the vertex $v_1$ instead of vertex $v$, and we have to blow up the exceptional divisor $E_{v_1}$ only $t_{v}-1$ times.

It means, that we are done by the induction hypothesis with claim 3).

Step 4)  Assume, that $h^1(\calO_Z) = r$, and furthermore, that there isn't a vertex $v \in |Z|$, such that $h^1(\calO_Z) = h^1(\calO_{Z - Z_v \cdot E_v})$.
With these properties $h^1(\calO_Z(Z'))$ is combinatorially computable:

We can again assume, $|Z|$ is connected.

We know, that $h^0(\calO_Z(Z')) = \max_{0 \leq l \leq Z} h^0(\calO_{Z -l}(Z'-l))$, and furthermore, 
there is an integer cycle $0 \leq l \leq Z$, such that $h^0(\calO_{Z}(Z')) = h^0(\calO_{Z-l}(Z'-l))$ and $H^0(\calO_{Z-l}(Z'-l))_{reg} \neq \emptyset$ or $l = Z$.

Now let's denote $I_l = |Z-l| \cap |Z'-l|$ and let's denote $Z_l = (Z-l)_{|Z-l| \setminus I_l}$.
If $H^0(\calO_{Z-l}(Z'-l))_{reg}$ is nonempty, then using Theorem \ref{relgen} in the case $\calv_1 = |Z_l|$ , $\calv_2 = I_l$, $Z_1 = Z_l$ and $Z_2 = Z-l - Z_l$, we get $h^0(\calO_{Z-l}(Z'-l)) = \chi(\calO_{Z-l}(Z'-l)) + h^1(\calO_{Z_l}(Z'-l))$ and $ h^0(\calO_{Z-l}(Z'-l)) \geq \chi(\calO_{Z-l}(Z'-l)) + h^1(\calO_{Z_l}(Z'-l))$ is always true.

This means $ h^0(\calO_Z(Z')) = \max_{0 \leq l \leq Z} (\chi(\calO_{Z-l}(Z'-l)) + h^1(\calO_{Z_l}(Z'-l)))$, and notice, that every term is computable on the right hand side, by previous cases, and by induction hypothesis on $h^1(\calO_Z)$.

Indeed if $I_l \neq \emptyset$, this is trivial because of $h^1(\calO_{Z_l}) < r$ and if $I_l = \emptyset$, then it follows from step 3).

This proves claim 4).

5) Assume, that $h^1(Z) = r$, then $h^1(\calO_Z(Z'))$ is combinatorially computable:

We can again assume, that $|Z|$ is connected.

We know, that $h^0(\calO_Z(Z')) = \max_{0 \leq l \leq Z} h^0(\calO_{Z -l}(Z'-l))$, and furthermore, 
there is an integer cycle $0 \leq l \leq Z$, such that $h^0(\calO_{Z}(Z')) = h^0(\calO_{Z-l}(Z'-l))$ and $H^0(\calO_{Z-l}(Z'-l))_{reg} \neq \emptyset $ or $l = Z$.

For all $0 \leq l \leq Z$, let's denote $I_l$ the smallest subset of $|Z-l|$, such that $h^1(Z-l) = h^1((Z-l)_{I_l})$, and let's denote $Z_l = (Z-l)_{I_l}$.
Now $h^0(\calO_{Z-l}(Z'-l)) \geq \chi(\calO_{Z-l}(Z'-l)) + h^1(\calO_{Z_l}(Z'-l))$, and  if $H^0(\calO_{Z-l}(Z'-l))_{reg}$ is nonempty or $l=Z$, then equality happens.

This means that $h^0(\calO_Z(Z')) = \max_{0 \leq l \leq Z} (\chi(\calO_{Z-l}(Z'-l)) + h^1(\calO_{Z_l}(Z'-l)))$ and the right hand side is computable by previous cases.

This proves our theorem completely.

\end{proof}

\end{document}